\documentclass[11pt]{amsart}
\pdfoutput=1 

\usepackage[
text={440pt,575pt},
headheight=9pt,
centering
]{geometry}

  \usepackage{natbib}

\usepackage[ruled]{algorithm2e}
\usepackage{hyperref}
\usepackage{xcolor}
\usepackage{appendix}

\definecolor{darkred}{RGB}{160,0,0}
\definecolor{darkblue}{RGB}{0,0,160}
\hypersetup{
  colorlinks,
  citecolor=darkblue,
  filecolor=black,
  linkcolor=darkblue,
  urlcolor=darkblue
}

\usepackage[utf8]{inputenc}
\usepackage[T1]{fontenc}
\usepackage{microtype}
\usepackage{lmodern}
\usepackage{amsmath,amssymb,amsthm}
\usepackage{soul}
\usepackage{mdwlist}
\usepackage[pdftex]{graphicx}
\usepackage{latexsym}
\usepackage{verbatim}
\usepackage{graphicx,caption,subcaption}
\usepackage{xcolor}
\usepackage{todonotes} 
\usepackage{enumitem}
\usepackage{upgreek}

\allowdisplaybreaks 

\theoremstyle{plain}
\newtheorem{theorem}{Theorem}[section]
\newtheorem{corollary}[theorem]{Corollary} 
\newtheorem{lemma}[theorem]{Lemma}

\theoremstyle{definition}
 	
\newtheorem{remark}[theorem]{Remark}

\newcommand{\R}{\mathbb{R}}

\newcommand{\N}{\mathbb{N}}
\newcommand{\PP}{\mathbb{P}}
\newcommand{\NN}{\mathrm{N}}

\newcommand{\D}{\mathcal{D}}
\newcommand{\F}{\mathcal{F}}

\newcommand{\HH}{\mathcal{H}}
\newcommand{\I}{\mathcal{I}}
\newcommand{\K}{\mathcal{K}}
\newcommand{\M}{\mathcal{M}}

\newcommand{\W}{\mathcal{W}}
\newcommand{\E}{\mathbb{E}}
\newcommand{\EE}{\mathcal{E}}
\newcommand{\Var}{\mathrm{Var}}
\newcommand{\Cov}{\mathrm{Cov}}

\newcommand{\dx}{\mathrm{d}x}
\newcommand{\dy}{\mathrm{d}y}
\newcommand{\dF}{\mathrm{d}F}
\newcommand{\dFhat}{\mathrm{d}\hat{F}}
\newcommand{\ddiff}{\mathrm{d}(F_0 - \hat{F}_0)}
\newcommand{\dt}{\mathrm{d}t}
\newcommand{\du}{\mathrm{d}u}
\newcommand{\dv}{\mathrm{d}v}
\newcommand{\dw}{\mathrm{d}w}

\newcommand{\nt}{\lfloor nt\rfloor}

\usepackage{mathtools}

\title{Testing for correct model specification in copula regression models}
\author{Holger Dette, Philip Dörr}
\date{\today}

\begin{document}

\begin{abstract}
We propose a goodness-of-fit test for semiparametric copula regression models. Such models express the regression function in terms of marginal distribution functions and copula densities and therefore provide a flexible way to avoid fully nonparametric estimation in high-dimensional regression problems. Their performance, however, depends crucially on the specification of the parametric copula family. Instead of testing the copula model itself, we assess misspecification directly at the level of the induced regression function. To this end, we introduce a weighted $L^2$-distance between the true regression function and its best approximation within the postulated copula regression model. A kernel-based estimator of this distance is proposed and shown to be consistent and asymptotically normal under both the null hypothesis of correct specification and fixed alternatives. We derive a classical specification test and, using a self-normalized sequential statistic, construct pivotal confidence intervals and tests for relevant deviations from the model. 
Finite-sample simulations demonstrate accurate level approximation and good power properties of the proposed procedures.
\end{abstract}

\maketitle


\section{Introduction}    \def\theequation{1.\arabic{equation}}	
   \setcounter{equation}{0}

Nonparametric regression models 
\begin{equation}
    Y = m(X) + \varepsilon\,, \label{pd1}
\end{equation}
with $d$-dimensional covariatee  and real-valued response $Y$ suffer from the curse of dimensionality, which severely deteriorates the performance of classical smoothing methods. To mitigate this difficulty, \citet{Noh2013} proposed a copula-based approach for the estimation of the unknown regression function $m$, which avoids direct nonparametric estimation of the conditional mean. Their method uses the decomposition 
\begin{equation}
    m(x) =\E(Y|X=x)= \frac{\E(Yc( F_0(Y), \F(x))}{c_X(\F(x))}\,, \label{2}
\end{equation} 
where $c$ and $c_X$ denote the copulas of the vectors $(Y,X^\top)^\top $ and $X=(X^{(1)}, \ldots, X^{(d)})^\top $, respectively, and $F_0$ and $\F = (F_1, \ldots, F_d)^\top$ are the (marginal) distribution functions of the response $Y$ and the predictor $X$, respectively. 
The regression function is then  estimated by combining nonparametric estimators of the marginal distributions with a parametric estimator of the copula. The resulting semiparametric {\it copula regression estimator}  inherits flexibility from the 
choice of the copula family and from the nonparametric estimation of the margins 
\citep{Noh2013}.

Since their seminal work, the idea of exploiting copula representations in regression modeling has been considered by many authors for  
different regression settings and data structures. 
For example, \citet{Noh2015} proposed a semiparametric estimator for conditional quantiles based on a copula representation of the conditional distribution function  \citep[see also][]{Kraus2017,Jobst2023}.  Copula regression for censored data, data with missing observations, and hierarchical data has been developed by \cite{Bouezmarni2020}, \cite{HAMORI2020}, and \cite{apkorivest2025}, and applications of the method can be found in \cite{Tran2022} and \cite{Bouezmarni2025}  among others. 
Recently, similar approaches have been considered in the context of {\it distributional copula regression}, where 
the dependence structure between variables can vary with covariates and is modeled through regression predictors for the parameters of the marginal distributions and the copula \citep[see, e.g.,][]{Klein2020,Klein2023}.

Despite their appealing flexibility, 
as pointed out by \cite{Noh2013} and  \citet{dette2014some},
copula-based regression models rely critically on the  correct specification of the copula.   The results in the last-named  paper demonstrate that copula misspecification can severely distort the estimated regression function, indicating that the copula specification is a fundamental modeling assumption.  In particular, 
statistical procedures for testing the validity of the copula specification are of central importance to ensure reliable inference. 
A substantial amount of literature in this direction has addressed the problem of testing the hypothesis
\begin{equation}
    H_0: c \in \{c(\cdot, \vartheta) \mid \vartheta \in \Theta\} \label{null}
\end{equation}
to check  adequacy of parametric copula specifications, and we refer to the reviews of \cite{Genest2009}, \cite{Berg2009}  and \cite{Fermanian2013} and the references therein. However, this formulation may be misleading, as using an incorrectly specified parametric copula $ c(\cdot, \vartheta)$ for the dependence structure does not necessarily result in a completely misspecified copula regression 
\begin{equation}
    m(x,\vartheta) =
    \frac{\E(Yc(F_0(Y), \F(x), \vartheta)}{c_X(\F(x), \vartheta)}\,, \label{2a}
\end{equation} 
for the regression function $m$ in \eqref{2}. In other words, $m(\cdot , \vartheta)$ might be a good approximation of $m$ although the parametric copula family has been misspecified.

In this paper, we therefore study the problem of misspecification in copula-based regression from this (different) perspective, focusing directly on the discrepancy between the regression function $m(\cdot , \vartheta)$ implied by the postulated model and the true regression function $m$, which reflects the primary interest of copula regression: a good approximation of the unknown regression model by a model which can be estimated without facing the curse of dimensionality. More specifically, we develop statistical inference for the deviation
\begin{equation}
    M^2 = \int \big (m(x) - m(x, \vartheta^* ) \big )^2 \pi (\dx)\,, \label{m^2}
\end{equation}
between  the regression function $m$ and its best approximation $m(\cdot , \vartheta^* )$ by a copula regression model of the form \eqref{2a} implied by  a given family of parametric copula models $\{c(\cdot, \vartheta) \mid \vartheta \in \Theta\}$. Here $\Theta \subseteq \R^l$ denotes the parameter space and $\pi $ is an appropriate measure defined below. We are particularly interested 
in the hypothesis 
\begin{align}
    \label{null1}
    H_0 : M^2 = 0\,, 
\end{align}
which  considers the specification problem of the copula directly at the level of implied regression functions by the copula regression approach. Moreover, when applying copula regression in practice, the assumption of a specific parametric copula family not necessarily reflects the belief that the copula precisely describes the dependence structure. Instead, it reflects the hope that the implied copula regression function is close to the true regression function $m$ so that reasonable estimation avoiding the curse of dimensionality is possible. Therefore, we are also interested  in confidence intervals for the measure of deviation between $m$ and its best approximation by a function of the form $m(\cdot , \vartheta) $, and in the hypothesis that the deviation is small, that is,
\begin{align}
    \label{null2}
      H_0 : M^2 \leq  \boldsymbol{\Delta} ~~{\rm versus ~~} H_1 : M^2 > \boldsymbol{\Delta}\,.
\end{align}
Here, $\boldsymbol{\Delta} > 0$ is a pre-specified constant for which one agrees that the deviation is sufficiently small such  that semiparametric estimation is reasonable. Note that hypotheses of the form \eqref{null2} have found considerable interest in the field of 
{\it  tolerant testing}, where one assesses whether the data is consistent with any distribution that lies within a given neighborhood of the candidate \citep[see, e.g.,][and the references therein]{canonne22a, kania2026testingimprecisehypotheses}.   Other recent references, where hypotheses of the form \eqref{null2} have been considered, are  \cite{wiedetal2024}, \cite{BastianDetteHeiny} and \cite{baillo2025bootstrap2025}.

The remaining part of this  article is organized as follows. In Section~\ref{sec2}, we give more details  on the copula regression approach as proposed by \cite{Noh2013}. We introduce an estimator of $M^2$ in Section~\ref{sec3}, which does not require direct estimation of the regression function  (thus avoiding the curse of dimensionality) and show its  asymptotic normality under 
the null hypothesis \eqref{null1}
in Section~\ref{sec3.1} and under the  alternative $M^2>0 $ in Section~\ref{sec3.2}, respectively.  These results can be used directly to develop an asymptotic level-$\alpha$ test for the hypothesis \eqref{null1} and to prove its consistency. However, inference based on these estimators such as the construction of confidence regions or tests of hypotheses of the form  \eqref{null2} is challenging because the asymptotic variance has a highly complex structure in the case $M^2 >0$. To address this problem, we develop  pivotal inference for the measure $M^2$ in Section~\ref{sec4}, which avoids estimation of said asymptotic variance. In Section~\ref{sec5}, we present the results of a small simulation study to demonstrate suitable finite sample behavior. Section~\ref{sec6} gives a short conclusion and outlook on subsequent research. The appendix section~\ref{appendix_a} gathers all proofs of our theoretical results, including some technical auxiliary lemmas whose proofs are gathered in the additional appendix section~\ref{appendix_b}. 

\section{Copula regression revisited} \label{sec2} 
   \def\theequation{2.\arabic{equation}}	
   \setcounter{equation}{0}

Let $(X_1,Y_1) , \ldots , (X_n,Y_n)$ denote a sample of i.i.d. random variables, where the $Y_i$, $i=1, \ldots, n,$ are univariate responses and $X_i = \bigl(X_{i}^{(1)}, \ldots, X_{i}^{(d)}\bigr)^\top$ are $d$-dimensional predictors with distribution function $F_0$
 and  $\mathcal{F}= (F_1, \ldots, F_d)^\top$, respectively. 
 We consider the nonparametric regression model 
%
\begin{equation}
    Y_i = m(X_i) + \varepsilon_i\,, \qquad i=1, \ldots, n\, , \label{pd2}
\end{equation} 
where $\varepsilon_i = Y_i - \E(Y_i \mid X_i)$ is centered and independent from $X_i$.
We tacitly assume that $X$ and $Y$ have Lebesgue densities.

Let $c$ denote the (joint) copula density of the vector $ (Y,X^\top )^\top = \bigl(Y, X^{(1)}, \ldots, X^{(d)}\bigr)^\top$ and let 
\[  
    c_X(u_1, \ldots, u_d) = \int_0^1 c(u_0, u_1, \ldots, u_d)\du_0
\]
denote the (inner) copula density of the covariates $X= \bigl(X^{(1)}, \ldots, X^{(d)}\bigr)^\top$. 
The approach of \cite{Noh2013} imposes a parametric  assumption on the copula density $c$, that is, 
\begin{equation}
    c \in \mathcal{ C} := \big \{c(\cdot, \vartheta) \mid \vartheta \in \Theta \big \}.
\label{model}
\end{equation}
We assume  identifiability, i.e., $c(\cdot , \vartheta') \neq c(\cdot, \vartheta)$ whenever $\vartheta' \not = \vartheta$, and denote by  $\vartheta_0 \in \Theta$ the parameter corresponding to $c$ if  \eqref{model} is satisfied, i.e., $c(\cdot ) = c(\cdot , \vartheta) $. In this case,  based on \eqref{2a}, the unknown regression function $m$ can be represented in the form
\begin{align}
\label{det1a}    
    m(x) = m(x, \vartheta_0)=\frac{1}{c_X(\F(x), \vartheta_0)}\int_\R yc(F_0(y), \F(x), \vartheta_0)\dF_0(y)\,.
\end{align}
To estimate $m$, we denote by  
\[
    \hat{F}_0 (y) = \frac{1}{n} \sum_{i=1}^n \textbf{1}\{ Y_i \leq y\} ~~~\text {  and ~~} 
    \hat{F}_j (x_j) = \frac{1}{n} \sum_{i=1}^n \textbf{1}\{ X_i^{(j)} \leq x_j\}
\]
the empirical distribution functions of the response and the $j$-th predictor, respectively, and define $\hat{\F} = (\hat{F}_1, \ldots, \hat{F}_d)$ as the vector of the marginal empirical distribution functions of the covariate.
For better readability, we occasionally use the short  notation
~$U_i := \F(X_i)$ for $i = 1, \ldots, d$. Additionally, we write $\hat{U}_j := \hat{\F}(X_j).$
Replacing the distribution functions $F_0, \F$ by their empirical counterparts $\hat{F}_0, \hat{\F}$ and the true parameter $\vartheta_0$ by an appropriate estimator $\hat{\vartheta}_n$, \cite{Noh2013} proposed the estimator
\begin{align}
\label{det1b}
    \hat{m}(x) := \frac{1}{c_X(\hat{\F}(x), \hat{\vartheta}_n)}\int yc(\hat{F}_0(y), \hat{\F}(x), \hat{\vartheta}_n)\dFhat_0(y) \,.
\end{align}
Then, under correct specification, these authors showed that $\sqrt{n}(\hat{m}(x) - m(x))$ converges weakly to a normal distribution.  

If \eqref{model} is not satisfied, then the asymptotic properties depend on the choice of the estimator $\hat \vartheta_n$. To be specific, in this paper we focus on the pseudo-maximum likelihood estimator 
\begin{equation}
    \hat{\vartheta}_{n} := \hat{\vartheta}_n^{\text{PL}} := \arg \max_{\vartheta \in \Theta} \sum_{i=1}^n \log\left(c(\hat{F}_{0}(Y_i), \hat{\F}(X_i), \vartheta)\right) \label{mpl}
\end{equation}
with respect to the joint copula density. Other estimators with similar properties could be considered as well, \citep[see Assumption B in][]{Noh2013}. If the model \eqref{model} is correctly specified, then \cite{tsukahara2005semiparametric} showed  that $\sqrt{n}(\hat{\vartheta}_n - \vartheta_0) = O_{\PP}(1)$, implying that $\hat{\vartheta}_n$ is a $\sqrt{n}$-consistent estimator of $\vartheta_0$. 
On the other hand, if  \eqref{model} does not hold, it can be shown that $\sqrt{n}(\hat{\vartheta}_n - \vartheta^*) = O_{\PP}(1)$  \citep[see Remark 3 in][]{Noh2013}, where $\vartheta^*$ corresponds to the parameter minimizing the Kullback-Leibler distance between the true copula and the class $\mathcal{C}$ in \eqref{model}, that is,
\begin{equation}
    \vartheta^* := \arg\min_{\vartheta \in \Theta} \int_{[0,1]^{d+1}} \log(c(u_0, \textbf{u})) - \log(c(u_0, \textbf{u}, \vartheta))\text{d}C(u_0, \textbf{u})\,.\label{theta_star}
\end{equation}
As pointed out in the same reference, the estimator \eqref{mpl} has the stochastic expansion 

\begin{align}
\label{item6}
    \hat{\vartheta}_n - \vartheta^* = n^{-1}\sum_{k=1}^n \eta_k + o_{\PP}(n^{-1/2})\,, 
\end{align}
where $\eta_1, \eta_2, \ldots$ are i.i.d.\ mean-zero random vectors defined by $\eta_k = \eta(F_0(Y_k), \F(X_k), \vartheta^*)$ for the functional
\begin{align*}
    \eta(U_0, U_1, \ldots, U_d, \vartheta^*) :=~& J^{-1}(\vartheta^*) \Biggl(\frac{\partial}{\partial_\vartheta} \log c(U_0, U_1, \ldots, U_d, \vartheta^*)  \\
    +~& \left.\sum_{j=0}^d \int_{[0,1]^{d+1}} (\textbf{1}\{U_j \leq v_j\} - v_j)\frac{\partial^2}{\partial \vartheta \partial v_j} \log c(v_0, \textbf{v}, \vartheta^*)\mathrm{d}C(v_0, \textbf{v}, \vartheta^*)\right), \\
    J^{-1}(\vartheta^*) :=~& \int_{[0,1]^{d+1}} \left(\frac{\partial^2}{\partial \vartheta \partial \vartheta^\top} \log c(u_0, \mathbf{u}, \vartheta^*)\right)\mathrm{d}C(u_0, \textbf{u}, \vartheta^*)\,.
\end{align*}

Throughout this paper, we use the abbreviation $m^*(\cdot) := m(\cdot, \vartheta^*)$ for the copula regression model induced by \eqref{det1} for the parameter $\vartheta^*$, and define $\Delta(x) := m(x) - m^*(x)$ as the  error between the true regression function $m$ and its  approximation from the copula regression model at the point $x$. 
If \eqref{model} holds, then $\vartheta_0 = \vartheta^*$ and $\Delta \equiv 0$.

We occasionally abbreviate $\hat{c}(u_0, \mathbf{u}) = c(u_0, \mathbf{u}, \hat{\vartheta}_n)$ and $\hat{c}_X(\mathbf{u}) = c_X(\mathbf{u}, \hat{\vartheta}_n)$. For the estimators $\hat{F}_0, \hat{\F}, \hat{m}, \hat{c}, \hat{c}_X,$ we suppress the dependence of $n$ in the notation.
We assume that the parametrized copula densities $c = c(u_0, \textbf{u}, \vartheta)$ and $c_X(\textbf{u}, \vartheta)$ as well as the function $\boldsymbol{e}(\textbf{u}, \vartheta)$ are twice continuously differentiable in all arguments. We denote the first-order partial derivatives by $\partial_{u_0}c = \partial c/\partial u_0$, $\partial_{u_i}c = \partial c/\partial u_i$ for $i = 1, \ldots, d$, as well as $\partial_{\textbf{u}}c = \left(\partial_{u_1}c, \ldots, \partial_{u_d}c\right)^\top$ and $\partial_{\vartheta}c = \partial c/\partial \vartheta = \left(\partial c/\partial \vartheta_1, \ldots, \partial c/\partial \vartheta_l\right)^\top$. The second-order derivatives of $c$ and the derivatives of $c_X$ are denoted in an analogous fashion.

\section{A measure of deviation and its stochastic properties 
} \label{sec3}
   \def\theequation{3.\arabic{equation}}	
   \setcounter{equation}{0}

\subsection{An $L^2$-distance for the deviation of the copula regression model}
\label{sec3.0}

To measure the deviation between the unknown regression function $m$ and the best approximation by the copula regression approach, we consider a (weighted) $L^2$-distance for the function $\Delta (\cdot )$. More specifically,  motivated by classical results on estimating the $L^2$-norm of a density \citep[see][]{hall1987estimation,bickel1988estimating}, we consider the statistic
\begin{align}
    \label{det1}
    W_n := \frac{1}{n(n-1)}\sum_{i\not= j} \frac{1}{h^d} K\left(\frac{X_{i} - X_{j}}{h}\right)e_ie_j\hat{c}_X(\hat{U}_i)\hat{c}_X(\hat{U}_j)\,,
\end{align}
where $K(\cdot)$ is a kernel function, $h$ is a bandwidth depending on the sample size $n$ and 
\[
    e_i := Y_i - \hat{m}(X_i) 
\]
denotes the residual obtained from a fit by the copula regression approach \eqref{det1b}. Note that the factors $\hat{c}_X(\hat{U}_j) = c_X(\hat{U}_j , \hat \vartheta_n)$ are introduced in \eqref{det1} to cancel the denominator $\hat{c}_X( \hat{\mathcal{F}} (X_i), \hat \vartheta_n) $ in the definition of $\hat m(X_i)$ by \eqref{det1b}. 
Intuitively, replacing the residuals $e_i$ with $\Delta(X_i)$ and replacing $\hat{c}_X(\hat{U}) = \hat{c}_X(\hat{\F}(X))$ with $c_X^*(\F(X)) ={c}_X(\mathcal{F}({X}_i), \vartheta^*)$, and calculating the expectation, the statistic $W_n$ can be considered as an estimate of the $L^2$-distance 
\begin{equation}
    M^2 = \int \Delta(x)^2 c_X^*(\F(x))^2p(x)^2\dx = \E\left(\Delta(X)^2c_X^*(\F(X))^2p(X)\right), \label{m2_full}
\end{equation}
where $p$ denotes the density of the predictor. This means that in \eqref{m^2}, we put $\pi(\dx) := c_X(\F(x), \vartheta^*)p^2(x)\dx$.
Our first main result precisely specifies these heuristic arguments. For this purpose, we impose the following technical assumptions.   
\begin{enumerate}[label=(R\arabic*)]
    \setlength{\itemsep}{1pt}
    \item \label{item1} The covariate $X$ has a differentiable density $p $ which is uniformly bounded, and so are its first-order derivatives. The function $x \to \sigma^4(x) := \E(Y^4 \mid X=x)$ is continuously differentiable and bounded by a measurable function $b$ such that $\E(b(X)^2) < \infty$. Further, $\E(m(X)^4) < \infty$. The error function $\Delta = m  - m^*$ is continuously differentiable and $L^2$-bounded. 
    \item \label{item2} The parameter space $\Theta$ is a compact subset of $\R^l$ and $\vartheta^*$ is an interior point of $\Theta$. 
    \item \label{item3} The copula densities $c(y, x, \vartheta),$ $\vartheta = (\vartheta_1, \ldots, \vartheta_l)^\top  \in \R^l$ are partially twice continuously differentiable in all arguments and satisfy
    \begin{align*}
        &\E\left(\sup_{\vartheta \in \Theta} Y^2 c(F_0(Y), \F(X), \vartheta)^2\right) < \infty\,, \\
        &\E\left(\sup_{\vartheta \in \Theta} Y^2 \partial_z c(F_0(Y), \F(X),  \vartheta)^2\right) < \infty\,, \\
        &\E\left(\sup_{\vartheta \in \Theta} Y^2 \partial_{z_1z_2}^2c(F_0(Y), \F(X), \vartheta)^2\right) < \infty
    \end{align*}
    for any $z, z_1, z_2 \in \{u_0, u_1, \ldots, u_d, \vartheta_1, \ldots, \vartheta_l\}$, as well as
    \[
       \E\left(\sup_{\vartheta \in \Theta} Y_{i_1}^2Y_{i_2}^2 c(F_0(Y_{i_1}), \F(X_{j_1}), \vartheta)^2 c(F_0(Y_{i_2}), \F(X_{j_2}), \vartheta)^2\right) < \infty 
    \]
    for all combinations $i_1, i_2, j_1, j_2 \in \{1,\ldots,4\}$.
    \item \label{item4} The conditions in \ref{item3} also hold true if $Y$ is replaced with $X$ and $c$ is replaced with $c_X$. Further, $\E(\sigma^4(X)p(X)^3c_X(\F(X))^4) < \infty$.
    \item \label{item5} The kernel $K $ is a bounded, continuous, and symmetric function with $\int K(u)\du = 1$. 
\end{enumerate}


\begin{theorem} \label{thm2.1}
Under the regularity conditions \emph{\ref{item1}--\ref{item5},} $W_n$ is a consistent and asymptotically unbiased estimator of $M^2,$ that is, $W_n \overset{\PP}{\longrightarrow} M^2$ and $\E(W_n) \longrightarrow M^2$ as $n \rightarrow \infty$.
\end{theorem}

Theorem \ref{thm2.1} justifies the use of $W_n$ as a measure of deviation between $m$ and the regression function induced by the copula regression model. In the following two subsections we derive two results which can be used for uncertainty quantification of this estimate. 

\subsection{Behavior under the null hypothesis} \label{sec3.1}

If the parametric family $\{c(\cdot, \vartheta)\}_{\vartheta \in \Theta}$ is well-specified, i.e., if \eqref{model} is true, then $nh^{d/2}W_n$ has a normal limit distribution. 

\begin{theorem} \label{thm3.1}
Let $h$ be chosen in a way such that $h = o(1),$ but $nh^d \longrightarrow \infty$ as $n \rightarrow \infty$. Under the null hypothesis and the regularity conditions \emph{\ref{item1}--\ref{item5},} we have for the statistic $W_n$ as given in \eqref{det1}\emph{:}
\begin{align}
\label{sigmahata}
    nh^{d/2}W_n \overset{\D}{\longrightarrow} \NN(0, \sigma_0^2)
\end{align}
with the asymptotic variance
\[
    \sigma_0^2 = 2\int K(u)^2\du \cdot \int \sigma^2(x)^2c_X(\F(x))^2 p(x)^2\dx\,,
\]
where $\sigma^2(x) := \Var(Y - m(x))$. The asymptotic variance $\sigma_0^2$ can be consistently estimated by 
\begin{equation}
    \hat{\sigma}_{0,n}^2 = \frac{2}{n(n-1){h^d}}\sum_{i\not= j} K^2\bigg (\frac{X_i - X_j}{h} \bigg) e_i^2e_j^2\hat{c}_X(\hat{U}_i)^2\hat{c}_X(\hat{U}_j)^2\,. \label{sigmahat}
\end{equation}
\end{theorem}

This weak convergence behavior can be employed to implement a goodness-of-fit test for the classical hypothesis \eqref{null1}. More specifically, it follows from \eqref{sigmahata} and \eqref{sigmahat} that 
\[
  \frac{nh^{d/2}}{\hat{\sigma}_{0,n}}W_n \overset{\D}{\longrightarrow} \NN(0,1),  
\]
and therefore we propose to reject the null hypothesis \eqref{null1} whenever
\begin{align}
    \label{testclass}
    nh^{d/2} W_n > \hat \sigma_{0,n} u_{1-\alpha}\,,
\end{align}
where $u_{1-\alpha}$ denotes the $(1-\alpha)$-quantile of the standard normal distribution. This yields an asymptotic level $\alpha$-test for the hypothesis in \eqref{null}, and by Theorem \ref{thm2.1} this test is also consistent. 

\subsection{Behavior under fixed alternatives} \label{sec3.2}

We now assume that the copula model is misspecified, i.e., there is no $\vartheta \in \Theta$ with $c = 
c(\cdot, \vartheta)$, and investigate the weak convergence of the statistic $W_n$.
More specifically, we show that under the alternative $H_1 \negmedspace: c \not \in \mathcal{C},$ the statistic $\sqrt{n}(W_n - \E(W_n))$ converges weakly to a centered normal limit distribution, and we derive an approximation for $\E(W_n) $. Recall the definition of the best-possible parameter $\vartheta^* \in \Theta$ as in \eqref{theta_star}. In what follows, we use the shorthand notation $\Delta^*(x) := \Delta(x)c_X^*(\F(x))$. 

\begin{theorem} \label{thm4.1}
    Under the aforementioned regularity assumptions \emph{\ref{item1}--\ref{item5},} the test statistic $W_n$ as in \eqref{det1} has the following asymptotic behavior under fixed alternatives: 
    \[
        \sqrt{n}\left(W_n - \mu_n\right) \overset{\D}{\longrightarrow} \NN\left(0, 4\bigl(\sigma_1^2 + \sigma_2^2\bigr)\right),
    \]
    where 
    \begin{align*}
        \mu_n &:= \iint K(u)\Delta^*(u+vh)\Delta^*(v)p(u+vh)p(v)\du\dv\,, \\
        \sigma_1^2 &:= \Var(\Delta^*(X_1)^2p(X_1))\,,
    \end{align*}
    and 
    \[
        \sigma_2^2 := \E\bigl(\M_1(X_2)\M_1(X_3)\Delta^*(X_2)\Delta^*(X_3)p(X_2)p(X_3) + \varepsilon_1^2\Delta^*(X_1)^2p(X_1)^2\bigr)\,,
    \]
    with 
\begin{align}
    \M_k(x) =~& c_X(\mathbf{u}, \vartheta^*) \left(\eta_k^\top \partial_{\vartheta} m^*(x) + \sum_{l=1}^d \Bigl(\mathbf{1}\{X_{k}^{(l)} \leq x_l\} - F_l(x_l)\Bigr)\partial_{x_l} m^*(x) \partial_{u_l} F_l^{-1}(u_l)\right) \nonumber \\
    &- \int \bigl(\mathbf{1}\{Y_k \leq y\} - F_0(y)\bigr)c(F_0(y), \F(x), \vartheta^*)\dy\,. \label{mkx}
\end{align}
\end{theorem}

If $h\to 0$, we have 
\begin{align} \label{hd3}
  \mu_n  \approx  M^2 = \int_{\R^d} \Delta(x)^2c_X^*(\F(x))^2p(x)^2\dx,
\end{align}
where $\Delta = m - m^*$.
Therefore, we can use  Theorem \ref{thm4.1} for uncertainty quantification, when one applies $W_n$ for the estimation of $M^2$, if the postulated assumption \eqref{model} of the copula is not satisfied. As pointed out in the introduction, this will usually be the case in applications as the assumption is commonly made in the hope that the implied copula regression function is close to the true regression function $m$, so that reasonable estimation avoiding the curse of dimensionality is possible. However, the estimation of the limiting variance in Theorem~\ref{thm4.1}, which would be required for such an approach, is extremely difficult and we do not discuss the construction of such estimators here. 
Instead, we will develop a self-normalized version of Theorem \ref{thm4.1}, which avoids this estimation.

 \section{Pivotal confidence intervals and testing of relevant hypotheses} 
 \label{sec4}
  \def\theequation{4.\arabic{equation}}	
   \setcounter{equation}{0}

Self-normalization is a widely used concept, which scales a statistic by a measure of its own variability, thereby avoiding explicit variance estimation. This leads to reliable  inference in situations, where the estimation of the asymptotic variance is difficult, see, e.g., \cite{lobato2001} or \cite{shao2010self}.  We will use this concept here to develop pivotal uncertainty quantification for the statistic $W_n$ using  a sequential version of Theorem~\ref{thm4.1}.  For this purpose,  we fix  $\delta \in (0,1)$ and define for $t \in [\delta, 1]$ by
 \begin{align*}
    \hat{F}_{0\nt}(y) &:= \frac{1}{\nt}\sum_{j=1}^{\nt} \textbf{1}\{Y_j \leq y\}\,, \quad \hat{\F}_{\nt}(x) := \frac{1}{\nt}\sum_{j=1}^{\nt} \textbf{1}\{X_j \leq x\}\,,
\end{align*}
the empirical distribution functions from the sample $(X_1, Y_1), \ldots, (X_{\nt}, Y_{\nt}),$ and by 
\[
    \hat{\vartheta}_{\nt} := \arg \max_{\vartheta \in \Theta} \sum_{i=1}^n \log\left(c(\hat{F}_{0{\nt}}(Y_i), \hat{\F}_{\nt}(X_i), \vartheta)\right) 
\]
the corresponding sequential estimator of the parameter $\vartheta$. With these notations we obtain a sequential estimator of the regression function
\[
    \hat{m}_{\nt} (x) := \frac{1}{c_X(\hat{\F}_{\nt}(x), \hat{\vartheta}_{\nt})}\int yc(\hat{F}_{0\nt}(y), \hat{\F}_{\nt}(x), \hat{\vartheta}_{\nt})\dFhat_{0\nt}(y) \, 
\]
and corresponding residuals 
\begin{align*}
    e_{i\nt} &:= Y_i - \hat{m}_{\nt}(X_i)\, \quad \quad  (i=1, \ldots , \nt ).
\end{align*}
Next, we define  $\hat{U}_{i\nt} = \mathcal{F}_{\nt} (X_i)$, and by 
\begin{equation}
    W_{\nt} := \frac{1}{\nt(n-1)h^d}\sum_{1\leq i\neq j\leq \nt} {K \Big ( {X_i - X_j \over h }\Big ) } e_{i\nt}e_{j\nt}c_X(\hat{U}_{i\nt}, \hat{\vartheta}_{\nt})c_X(\hat{U}_{j\nt}, \hat{\vartheta}_{\nt}) \label{wnt}
\end{equation}
we obtain a sequential analogue of the statistic $W_n$ in \eqref{det1}.

The assumption $t \in [\delta,1]$ keeps $t$ bounded away from the origin, which is useful to avoid degenerate cases.

\begin{theorem} \label{thm4.6}
Assume that  the conditions of Theorem~\emph{\ref{thm4.1}} are satisfied and define  for $\delta \in (0,1)$\emph{:} 
\begin{align} \label{hd2a}
    \W_n &:= \int_{\delta}^1 |W_{\nt} - tW_n|\dt\,,
\end{align}
where $W_{\nt}$ is given by \eqref{wnt}. Then: 
\begin{align}
    \label{hd2}
{W_n - \mu_n \over \W_n}
\overset{\D}{\longrightarrow} 
W := \frac{B(1)}{\int_\delta^1 |B(t) - tB(1)|\dt}\,,
\end{align}
where $\{B(t)\}_{t \in [0,1]}$ denotes a standard Brownian motion. 
\end{theorem}

We now discuss two statistical consequences of Theorem \ref{thm4.6}. For this purpose, let  $q_\alpha(W)$ denote the $\alpha$-quantile of the distribution of the random  variable $W$ defined in \eqref{hd2} and consider the interval 
\begin{align}
    \label{conf}
  \hat I_n := [W_n - q_{1-\alpha/2}(W)\W_n,\,\, W_n + q_{1-\alpha/2}(W)\W_n]\,.
\end{align}
\begin{corollary}
\label{thm4.6a}
For $\alpha \in (0,1),$  the interval $\hat I_n$ defines a pivotal and asymptotic $(1-\alpha)$-confidence interval for the bias $\mu_n$ appearing in Theorem \emph{\ref{thm4.1}}. 
\end{corollary} 

\begin{remark} \label{rem4.3} Under appropriate conditions to the regression framework and the kernel, the above interval $\hat{I}_n$ is an asymptotic $(1- \alpha)$-confidence interval for the measure of deviation $M^2$ itself. If the bandwidth $h$ is chosen by the bias-variance tradeoff rule of thumb $h = \Theta(n^{-1/(4+d)}),$ then, in order to make the discrepancy $\sqrt{n}(\mu_n - M^2)$ vanish, we need a kernel of order $r > 3 + d/2$ and we require differentiability up to order $2 + \lceil d/2\rceil$ for the function $\Delta,$ the inner copula density $c_X^*$, and the density $p$ of $X$. Indeed, this is a consequence of a Taylor expansion  
\begin{align*}
    \mu_n - M^2 &= \iint K(u)(g(u + vh) - g(v))g(v)\du\dv \\
    &= h \int K(u)u\du \int g(v)g'(v)\dv + \frac{h^2}{2} \int K(u)u^2\du \int g(v)g''(v)\dv \\
    & \quad + \ldots + h^{r-1} \int K(u)u^{r-1}\du \int g(v)g^{(r-1)}(v)\dv + O(h^r)\,,
\end{align*}
where we used the notation $g(x) := \Delta^*(x)p(x)$. The order of the kernel ensures that aside from the $O(h^r)$ remainder term,  all terms vanish. The requirement $h^r = o(n^{-1/2}),$ employing the rule of thumb $h = \Theta(n^{-1/(4+d)}),$ yields the previously stated kernel order and differentiability assumptions.
\end{remark}

Finally, we consider the problem of testing the relevant hypotheses
\begin{subequations}
\begin{align}
\label{hd4}     
    H_0 : M^2 \leq \boldsymbol{\Delta} \text{~~ versus ~~} H_1: M^2 > \boldsymbol{\Delta}\,,
\end{align}
where $\boldsymbol{\Delta}$ denotes a prespecified  threshold. We propose to reject the null hypothesis in \eqref{hd4} whenever
\begin{align}
\label{hol12}
    W_n > \boldsymbol{\Delta} + q_{1-\alpha}(W) \W_n\,.
\end{align}
Similarly, for testing the hypotheses 
\begin{align} \label{hd6}
    H_0 : M^2 > \boldsymbol{\Delta} \text{~~versus ~~} H_1: M^2 \leq \boldsymbol{\Delta}\,,
\end{align}  
we propose to reject the null hypothesis whenever 
\begin{align*} \label{hd7}
    W_n \leq \boldsymbol{\Delta} + q_{\alpha}(W) \W_n\,.
\end{align*} 
\end{subequations}

\begin{corollary} \label{correjrate}
Let the assumptions of Theorem ~\emph{\ref{thm4.6}} and Remark~\emph{\ref{rem4.3}} be satisfied.
\begin{subequations}
\begin{itemize}
    \item [(a)] For the  test \eqref{hd4} we have 
    \begin{align}
    \lim_{n\rightarrow\infty}
    \mathbb{P }\big( W_n > \boldsymbol{\Delta} + q_{1-\alpha}(W) \W_n \big) = \begin{cases}
    1,      & \text{if } M^2  >  \boldsymbol{\Delta} \\ 
    \alpha, & \text{if } M^2  = \boldsymbol{\Delta} \\
    0,      & \text{if } M^2 < \boldsymbol{\Delta}
    \end{cases}.
    \end{align}
    \item[(b)] For the test \eqref{hd6} we have 
    \begin{align}
    \lim_{n\rightarrow\infty}
    \mathbb{P }\big(W_n \leq \boldsymbol{\Delta} +  q_{\alpha}(W) \W_n \big) =     
    \begin{cases}
1,      & \text{if } M^2 < \boldsymbol{\Delta} \\ 
\alpha, & \text{if } M^2 = \boldsymbol{\Delta} \\
0,      & \text{if } M^2 > \boldsymbol{\Delta}
    \end{cases}.
    \end{align}
\end{itemize}
\end{subequations}
\end{corollary}

\begin{remark} \label{rem4.4}  ~~~
\begin{itemize}
\item[(a)] We give a heuristic argument that the tests \eqref{hd4} and \eqref{hd6} are not very sensitive with respect to the value of $\delta$, which has been introduced in the statistic $\mathcal{W}_n$ to achieve numerical stability. For this purpose we introduce the notation ${V}_\delta = \int_\delta^1 |B(t) - t B(1)|\dt$ for the denominator of the random variable $W$ in \eqref{hd2} and make the dependence of the quantile $q_{1-\alpha}$ on the distribution of ${W} = {B}(1)/ V_\delta$ more explicit by using the notation $q_{1-\alpha} ({W}) $.
With these notations we obtain for the probability of rejection of the test \eqref{hd4}: 
\begin{align}  
\mathbb{P}\big(W_n > M^2 + q_{1-\alpha} ( {W} )  \mathcal{W}_n \big) 
& \approx 
 \mathbb{P}\left({B}(1) > \frac{\sqrt{n} 
  (\boldsymbol{\Delta} - M^2)}{2(\sigma_1^2 + \sigma_2^2)^{1/2}}  + {V }_\delta \cdot q_{1-\alpha}({B}(1)/ {V}_\delta ) \right),  \label{h40a}
\end{align}
where we have used the weak convergence 
$  \W_n  \overset{\D}{\longrightarrow}  2(\sigma_1^2 + \sigma_2^2)^{1/2} V_\delta $  for the approximation of the probabilities (this is shown in the proof of Theorem \ref{thm4.6}). Note that within the right-hand side of \eqref{h40a}, the constant $\delta $  appears only in the quantity $V_\delta  \cdot q_{1-\alpha}( {B}(1)/V_\delta)$.
However, for fixed $v>0$, we have $ {v }\cdot q_{1-\alpha}({B}(1)/v) = q_{1-\alpha}({B}(1))$. Therefore we expect that the probability in  \eqref{h40a} is not very sensitive with respect to the choice of $\delta$. 

Similar arguments can be made for properties of the test \eqref{hd6} and the  coverage probabilities of the confidence intervals \eqref{conf}, which also turn out to be relatively robust with respect to the choice of $\delta$. 
\item[(b)] We can replace the Lebesgue measure in the statistic \eqref{hd2a} by an arbitrary measure $\nu$
with  compact support contained in the interval $(0,1]$. Then the same arguments given in the proof of Theorem \ref{thm4.6} show that 
\[
    \frac{W_n - \mu_n}{\int_{\delta}^1 |W_{\nt} - tW_n|\mathrm{d}\nu(t)} \overset{\D}{\longrightarrow} \frac{B(1)}{\int_{\delta}^1 |B(t) - tB(1)|\mathrm{d}\nu(t)}\,, 
\]
and, arguing as in part (a), we conclude that the resulting test is not very sensitive with respect to the choice of the measure $\nu$. This remark is of importance, because we can obtain a  computationally tractable test statistic by choosing a discrete measure with a small number of support points.

\end{itemize}
\end{remark}

\section{Finite sample properties} \label{sec5}
 \def\theequation{5.\arabic{equation}}	
   \setcounter{equation}{0}

In this section we provide a small simulation study illustrating the finite sample properties of the tests proposed in Section \ref{sec3} and \ref{sec4}.
All results are based on $N = 1000$ simulation runs and the significance level for all tests is $\alpha = 0.05$.  The computation of the test statistic $W_n$ is achieved with help of the \texttt{R} package \texttt{Copula} (\cite{hofert2014package}). While for the theoretical investigations we originally defined the marginal distribution estimators $\hat{F}_0(y) = n^{-1}\sum_{i=1}^n \textbf{1}\{Y_i \leq y\},$ $\hat{F}_j(x_j) = n^{-1}\sum_{i=1}^n \textbf{1}\{X_i^{(j)} \leq x_j\}$ $(j = 1, \ldots, d)$, we now use the numerically more stable equivalents 
\begin{align*}
    \hat{F}_0(y) &:= \frac{1}{n+1} \sum_{i=1}^n \textbf{1}\{Y_i \leq y\}\,, & \hat{F}_j(x_j) &:= \frac{1}{n+1} \sum_{i=1}^n \textbf{1}\{X_i^{(j)} \leq x_j\}\,.
\end{align*}
For the kernel we used the Epanechnikov kernel in Subsection~\ref{sec5.1}, and a fifth-order version of the Gaussian kernel in Subsection~\ref{sec5.2}.

\subsection{Testing the classical null hypothesis} \label{sec5.1}

We verify the  finite sample properties of the  test  \eqref{testclass} for the classical hypothesis \eqref{m^2} for some of the common copula families, where  we replicate the data generating procedures in the simulation study of \cite{Noh2013}. More specifically, we generate samples from the Farlie-Gumbel-Morgenstern (FGM) copula, the Clayton copula, the Gaussian copula, and the $t$-copula with parameters listed in Table~\ref{tab:fgm}, where we treat all $t$-copulas with a fixed number of degrees of freedom as an individual one-parameter family. Thus we assume that there is no error in model \eqref{pd2} and specify the marginal distributions of the response $Y$ and the covariate $X$ as follows:
\begin{itemize}
    \item In case of the FGM copula, we take $Y \sim \NN(0,1)$ and we take $F(x) = 1 - \exp(-\exp(x))$ for the law of $X$. 
    \item In all other cases, we take $Y \sim \NN(1,1)$ and $X \sim \NN(0,1)$.
\end{itemize}
We consider different sample sizes $n \in \{100, 200, 500\}$ and varying parameter choices for the copula families. The bandwidth is chosen adaptively based on leave-one-out cross validation from the grid $\{0.2 + k/20 \mid k = 0, \ldots, 40\}$.
Table \ref{tab:fgm} documents the finite sample performance of the  test \eqref{testclass}  for the classical null hypothesis \eqref{m^2} when the model is well-specified ($M^2=0$). We observe that the nominal level is well approximated  for all  copula families under consideration. 

\begin{table}[h]
    \renewcommand{\arraystretch}{1.3}
    \centering
    \begin{tabular}{c|c|c|c|c}
        Copula & Parameters & $n=100$ & $n=200$ & $n=500$ \\
        \hline \hline
        FGM & $\vartheta = 0.5$ & 4.7 & 4.1 & 4.3 \\
        \hline
        Clayton & $\vartheta = -0.5$ & 4.8 & 3.6 & 5.8 \\
        \hline
        & $\vartheta = 0.1$ & 4.6 & 5.1 & 3.2 \\
        \hline
        & $\vartheta = 2$ & 4.5 & 5.8 & 4.1 \\
        \hline
        & $\vartheta = 8$ & 6.5 & 4.7 & 3.7 \\
        \hline
        Gaussian & $\rho = 0.6$ & 4.5 & 4.4 & 4.4 \\
        \hline
        & $\rho = -0.2$ & 3.4 & 4.1 & 3.6 \\
        \hline
        $t$ & $\rho = 0.6, df = 4$ & 4.7 & 3.9 & 4.6 \\
        \hline 
        & $\rho = 0.6, df = 40$ & 3.5 & 4.5 & 3.3
    \end{tabular}
    \caption{\it  Empirical rejection rates  of the test \eqref{testclass} under the null hypothesis $H_0: M^2=0$.}
    \label{tab:fgm}
\end{table}

To evaluate the finite simple performance of the test \eqref{testclass} under the  alternative, we consider the quadratic regression model
\begin{equation} \label{parabola}
    Y_i = \left(\frac{1}{2} - X_i\right)^2 + \varepsilon_i~~~~~~, ~~i=1 \ldots , n, 
\end{equation}
where $\varepsilon_1 , \ldots ,  \varepsilon_n$ are independent and $ \NN(0,1/100)$ distributed and $X_1, \ldots, X_n$ are independent $U(0,1)$ distributed.
The results are displayed in  Table \ref{tab:alt} and we observe that the test \eqref{testclass} 
reliably rejects the null hypothesis for all cases  under consideration. 

\begin{table}[]
    \renewcommand{\arraystretch}{1.3}
    \centering
    \begin{tabular}{c|c|c|c}
        copula & $n=100$ & $n=200$ & $n=500$ \\ \hline \hline
         Clayton & 56.2 & 89.4 & 100 \\ \hline
         Gaussian & 59.8 & 93.8 & 100 \\ \hline
         FGM & 57.9 & 93.5 & 100 \\ \hline
         $t$, $df=4$ & 58.6 & 94.3 & 100 \\ \hline
         $t$, $df=40$ & 57.7 & 92.9 & 100  
    \end{tabular}
    \caption{\it Empirical rejection rates of the test \eqref{testclass} under the alternative $H_0: M^2>0$, given the setting of \eqref{parabola}.}
    \label{tab:alt}
\end{table}

\subsection{Finite sample performance for distance estimation and relevant hypotheses} \label{sec5.2}

To employ the inference procedures introduced in Section~\ref{sec4}, we first determine the quantiles of the limiting random object $W$ in \eqref{hd2}. In the spirit of Remark~\ref{rem4.4}, we choose $\delta = 1/2$ and replace the Lebesgue measure in the statistic \eqref{hd2a} and in the denominator of the limiting distribution \eqref{hd2} by a uniform distribution supported at $10$ points $t = 0.5, 0.55, \ldots, 0.95$.
We then calculated the  $95$\%-quantile of the distribution of the random variable $W$  in \eqref{hd2} by $10^7$ simulation runs as  $q_{0.95} =  6.97836 
$. 

For all cases of interest, the measure of deviation $M^2$ as given in \eqref{hd3} has to be computed numerically, even in the single-covariate setup. Here the inner copula density $c_X$ reduces to $c_X \equiv 1,$ implying
\begin{align}
\label{hol11}
    M^2 = \int_{\R} (m(x) - m^*(x))^2 p(x)^2 \dx\,.
\end{align}
We consider the following single-covariate setups: 
\begin{itemize}
    \item[(S1)] We take $X \sim U(0, 6)$ and $Y = (4 - X)^2 + \varepsilon$, with $\varepsilon \sim \NN(0, 1/25)$, and fit this model to the family of Gaussian copulas.
    \item[(S2)] We take $X \sim \NN(0, 1)$ and $Y = (1/2 - X)^2 + \varepsilon$, with $\varepsilon \sim \NN(0, 1/100)$, and fit this model to the family of FGM copulas.
\end{itemize}

In scenario (S1), the true copula density $c$ between $X \sim U(0,6)$ and $Y = (4 - X)^2 + \varepsilon$ cannot be explicitly represented. Therefore, we determine the  parameter $\vartheta^*$ in \eqref{theta_star} corresponding to the best approximation by a Gaussian copula family as the pseudo-ML estimator $\hat{\vartheta}_{N_0}^{\mathrm{PL}}$ in \eqref{mpl} for a large initial sample reference sample of observations ($N_0 = 10^6$). The resulting estimate is $\vartheta^* = -0.6387,$ with a standard error less than $0.001$. Next, we calculate an approximation $\tilde F_0$ for the distribution function $F_0$ from  the reference sample. The ``best approximating'' regression is finally calculated by Monte Carlo integration, that is,  
\begin{align}
    m^*(x) = \int_0^1 F_0^{-1}(u) c_G(u, F(x), \vartheta^*) \du 
    &\approx \frac{1}{N_1} \sum_{i=1}^{N_1} \tilde {F}_0^{-1}(V_i) c(V_i, F(x), \vartheta^*)\,, \nonumber
\end{align}
where $c_G(\cdot, \rho)$ denotes the bivariate Gaussian copula density with correlation $\rho$,  and  $V_1,$ $V_2, \ldots, V_{N_1}$ are independent $U(0,1)$-distributed random variables ($N_1 = 5 \cdot 10^5$). 
Finally, we obtain the estimate of $M^2$ again by Monte Carlo integration, that  is 
\[
    M^2 \approx 6 \cdot \frac{1}{N_1} \sum_{i=1}^{N_1} \Bigl(m(U_i) - m^*(U_i)\Bigr)^2p(U_i)^2\,,
\]
where  $U_1, U_2, \ldots, U_{N_1} $ are independent random variables uniformly distributed on the interval $[0,6]$. The resulting value is given by  
$
    M^2 \approx 0.9144\,.
$
In Figure~\ref{fig_unif_samplesizes}, we display  the empirical rejection rates of the test \eqref{hol12}
for the hypotheses \eqref{hd4} for various values of $\boldsymbol{\Delta}$,
where the sample sizes are $n \in \{200, 500, 1000\}$.
To reduce the bias between $M^2$ and  $\mu_n$ as given in Theorem~\ref{thm4.1}, we employ a fifth-order version of the Gaussian kernel and for the bandwidth we use $h = 0.5 \cdot n^{-1/5}$, which turned out as a good rule of thumb. We observe that the results reflect the qualitative properties stated in Corollary  \ref{correjrate}. If $M^2 \approx \boldsymbol{\Delta}$, the rejection probability is close to the nominal level $\alpha$ and this approximation improves slightly with an increasing sample size. If $M^2 > \boldsymbol{\Delta},$ the rejection rates exceed the nominal level and increase with the sample size and the distance $M^2-\boldsymbol{\Delta}$. If $M^2 < \boldsymbol{\Delta}$, the rejection are very close to zero.

In order to investigate the sensitivity of the results with respect to the choice of the smoothing parameter, we display in Figure~\ref{fig_unif_bandwidths}  the empirical rejection rates of the test \eqref{hol12} for the hypotheses \eqref{hd4} for the bandwidths $h=c\cdot n^{-1/5}$ for various values of $c \in \{0.1, 0.5, 1\} $, where the sample size is given by $n=500$. We observe that  the test is rather robust with respect to different values of $c$.


\begin{figure}[h] 
\begin{center}
    \includegraphics[width=12cm, height=6cm]{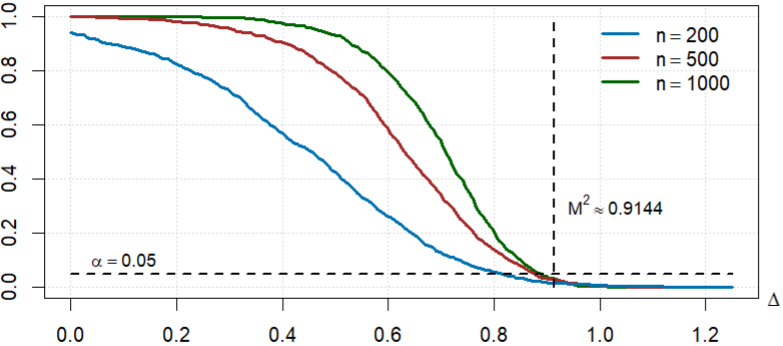}
\end{center}
\caption{\it Empirical rejection rates of the 
test \eqref{hol12}
for the hypotheses \eqref{hd4}  for various values of $\boldsymbol{\Delta}$ in scenario (S1).
The bandwidth is chosen as  $h = 0.5 \cdot n^{-1/5}$. The horizontal and vertical line display the nominal level  and the numerical approximation for the  deviation measure $M^2,$ respectively.} 
\label{fig_unif_samplesizes}
\end{figure}

\begin{figure}[h]
\begin{center}
    \includegraphics[width=12cm, height=6cm]{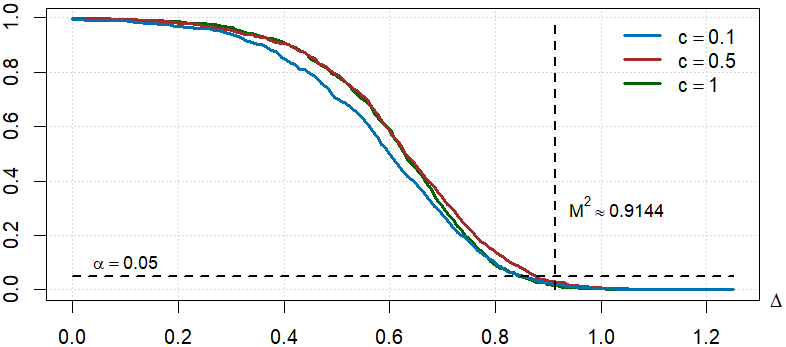}
\end{center}
\caption{
\it Empirical rejection rates of the 
test \eqref{hol12} for the hypotheses \eqref{hd4} for various values of the bandwidth $h = c \cdot n^{-1/5}$ 
in scenario (S1).
The sample size is $n=500$. The horizontal and vertical line display the nominal level and the numerical approximation for the deviation measure $M^2,$ respectively. }
\label{fig_unif_bandwidths}
\end{figure}

We continue considering the scenario (S2). When fitting a copula regression model with the FGM copula family, one often experiences the problem that the best-possible parameter $\vartheta^*$ is a boundary point of the parameter space $\Theta = [-1,1]$, 
as the FGM family only covers a small range of values of a dependence measure. For example, for Kendall's $\tau$, only values within the interval $[-2/9, 2/9]$ can be obtained if the FGM copula is used for modeling the dependency between the response and the covariate. 
If the dependency  between the response and the covariate is much stronger, it is therefore expected that $\vartheta^* = -1$ or $\vartheta^* = 1$. This is also the case for the parabolic model (S2), which means that assumption \ref{item2} is not satisfied.  Nevertheless, our procedure still proves to be robust and yields valid estimates for $M^2$.
According to equation (3) in \cite{Noh2013}, the FGM family admits a slightly less implicit representation of $m^*,$ namely,  
\[
    m^*(x) = \E(Y) + \vartheta(2F(x) - 1) \int F_0(y)(1 - F_0(y))\dy\,,
\]
and we determine an approximation $\tilde {F}_0$ for $F_0$ as described above, which yields 
$M^2 \approx 0.245$. The empirical rejection probabilities of the  test \eqref{hol12} for the hypotheses \eqref{hd4} are displayed in 
Figure~\ref{fig_fgm_samplesizes}, where the bandwidth is chosen again as $h = 0.5 \cdot n^{-1/5})$. We observe similar properties as in scenario (S1). 

\begin{figure}[h]
\begin{center}
    \includegraphics[width=12cm, height=6cm]{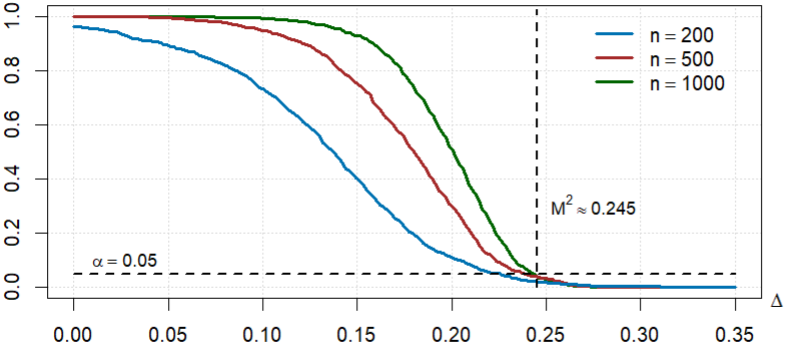}
\end{center}
\caption{
\it Empirical rejection rates of the 
test \eqref{hol12}
for the hypotheses \eqref{hd4}  for various values of $\boldsymbol{\Delta}$ in scenario (S2). The bandwidth is chosen as  $h = 0.5 \cdot n^{-1/5}$. The horizontal and vertical line display the nominal level  and the numerical approximation for the  deviation measure $M^2,$ respectively. \label{fig_fgm_samplesizes}}
\end{figure}

We conclude this section considering an example with a multivariate predictor: 
\begin{itemize}
    \item[(S3)] We take $X = \big (X^{(1)}, X^{(2)}, X^{(3)}\big )$ as a 3-dimensional vector from a  Gaussian distribution with mean zero and covariance matrix 
    \[
        \Sigma = \begin{pmatrix} 1 & 0.5 & 0.5 \\ 0.5 & 1 & 0.5 \\ 0.5 & 0.5 & 1 \end{pmatrix},
    \]
    and define  $Y$ by the equation $Y = m(X) + \varepsilon$, where  $ \varepsilon \sim \NN(0, 1/100)$ and for $x = (x_1, x_2, x_3) \in \R^3$: 
    \[
        m(x) = (1 + x_1 - x_2 + x_3)^2\,.
    \]
We consider the  Frank copula family for the copulas regression approach.
\end{itemize}
In setups involving covariates of dimension $d\geq 2$,  the inner copula density $c_X$ of the covariates is not constant and  influences the test statistic and the measure of deviation as well. This complicates the numerical calculation of $M^2$ substantially. 
For the scenario (S3) we obtain $\vartheta^* \approx 2.3438$ for the parameter corresponding to the best approximation 
with corresponding distance $M^2\approx 0.474$
when using the Frank copula family. 
The empirical rejection probabilities of the 
test \eqref{hol12}
for the hypotheses \eqref{hd4}  
are displayed in Figure~\ref{fig_frank_500}  for the sample sizes $n \in  \{200, 500, 1000\} $, where  the bandwidth was chosen as $h = 1 \cdot n^{-1/7}$. We observe the  same qualitative properties  as in the previous examples (see also Corollary \ref{correjrate}), where the approximation of the nominal level $\alpha = 0.05$ at the ``boundary''  $M^2 \approx \boldsymbol{\Delta}$ is improving with an increasing sample size. Finally, we consider the sensitivity of the test with respect to choice of the smoothing parameter and  display in Figure \ref{fig_frank_1000} empirical rejection probabilities of the 
test \eqref{hol12}
for the hypotheses \eqref{hd4}, where the bandwidth is chosen as $h =  c \cdot n^{-1/7}$ with  $c \in \{ 0.5,1,2\} $ and  sample size  $n=1000$. While the differences between the choices $c=1$ and $c=2$ are relatively moderate, the results for the choice $c=0.5$ indicate that the test is more sensitive with respect to the choice of the bandwidth if the dimension $d$ of the predictor is larger than $2$. In general, we strongly recommend to apply the bias-variance tradeoff rule-of-thumb $h = 1 \cdot n^{-1/(4+d)}$.

\begin{figure}[h]
\begin{center}
        \includegraphics[width=12cm, height=6cm]
        {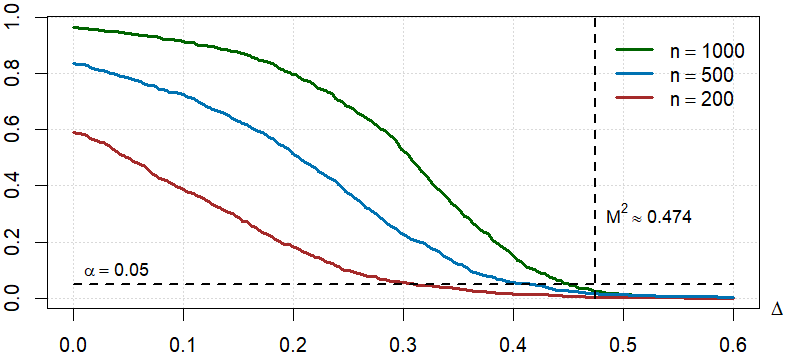}
\end{center}
\caption{\it Empirical rejection rates of the 
test \eqref{hol12}
for the hypotheses \eqref{hd4}  for various 
values of $\boldsymbol{\Delta}$ in scenario (S3).
The bandwidth is chosen as  $h = 1 \cdot n^{-1/7}$.  The horizontal and vertical line display the nominal level  and the numerical approximation for the  deviation measure $M^2$, respectively.} \label{fig_frank_500}
\end{figure}

\newpage

\begin{figure}[h]
\begin{center}
       \includegraphics[width=12cm, height=6cm]{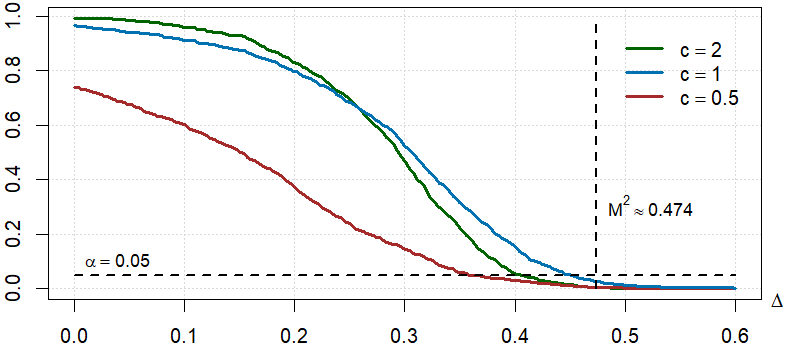}
\end{center}
\caption{
\it Empirical rejection rates of the 
test \eqref{hol12} for the hypotheses \eqref{hd4}  for various values of the bandwidth $h = c \cdot n^{-1/7}$ in scenario (S3). The  sample size is $n = 1000$. The horizontal and vertical line display the nominal level and the numerical approximation for the deviation measure $M^2,$ respectively.} \label{fig_frank_1000}
\end{figure}

\section{Conclusion and outlook} \label{sec6}

We have proposed inference procedures for assessing the adequacy of semiparametric
copula regression models. The central idea is to evaluate misspecification not at the
level of the copula itself, but at the level of the regression function induced by the
postulated copula family. For this purpose, we introduced a weighted $L^2$-distance
between the true regression function and its best approximation within the copula
regression model. A kernel-based estimator of this distance was shown to be consistent
and asymptotically normal both under correct specification and under fixed alternatives.
These results yield a classical goodness-of-fit test for exact specification.
\\
More importantly, the proposed framework also allows us to address relevant
misspecification. This is particularly useful in applications, where the parametric
copula family is typically not expected to describe the full dependence structure exactly,
but is used as a parsimonious device for approximating the regression function. To avoid
the difficult estimation of the asymptotic variance under alternatives, we developed a
self-normalized sequential statistic. This leads to pivotal confidence intervals for the
deviation measure and to tests for relevant hypotheses of the form that the model
deviation is below or above a prespecified tolerance level. The simulation results indicate
that these procedures have satisfactory finite-sample properties and, in particular,
provide a reasonable approximation of the nominal level at the boundary of the relevant
hypotheses.
\\
The methodology developed here suggests several directions for future research. Since
the statistic is based on kernel smoothing, its finite-sample performance can be sensitive
to the bandwidth choice, especially when the dimension of the covariate vector increases.
It would therefore be interesting to develop related procedures which do require the specification of smoothing parameters. Another promising direction is the extension to weakly dependent observations,
multi-response regression models, and regression functionals beyond the conditional mean.
Finally, the proposed approach could be adapted to variable-selection problems in copula
regression, where the goal is to identify those covariates that contribute significantly to
the induced regression function. \par 
\medskip 
\begin{appendices}
\section{Proofs of Main Theorems} \label{appendix_a}
 \def\theequation{A.\arabic{equation}}	
   \setcounter{equation}{0}

In this appendix section we give the proofs of the main theorems~\ref{thm2.1},~\ref{thm3.1},~\ref{thm4.1} and~\ref{thm4.6}. Recall the representations of $m(x)$ and $\hat{m}(x)$ as given in \eqref{det1a}, \eqref{det1b}. For the numerators, we introduce the shorthand notations
\begin{subequations}
\begin{align}
    \boldsymbol{e}(\F(x)) & := \E\bigl(Yc(F_0(Y), \F(x), \vartheta_0)\bigr) \,, \label{e(x)} \\   \hat{\boldsymbol{e}}(\hat{\F}(x)) &:= \frac{1}{n} \sum_{k=1}^n Y_kc(\hat{\F}_0(Y_k),\hat{\F}(x),\hat{\vartheta}_n) \,, \label{e(x)hat}
\end{align}
\end{subequations}
so that
\begin{align*}
    m(x) & = \frac{\boldsymbol{e}(\F(x))}{c_X(\F(x))}\,, & \hat{m}(x) &= \frac{\hat{\boldsymbol{e}}(\hat{\F}(x))}{\hat{c}_X(\hat{\F}(x))}\,.
\end{align*}

Moreover, in all of what follows, we use the abbreviation $K_{ij} := K((X_i - X_j)/h)$ for any two indices $i,j \in \{1, \ldots, n\}$, $i \neq j$.

\subsection*{Proof of Theorem~\ref{thm2.1}}

We decompose the test statistic $W_n$ as follows:
\begin{align}
    W_n &= \frac{1}{n(n-1)}\sum_{i\not= j}\frac{K_{ij}}{h^d}e_ie_j\hat{c}_X(\hat{U}_i)\hat{c}_X(\hat{U}_j) \nonumber \\
    &= \frac{1}{n(n-1)}\sum_{i\not= j}\frac{K_{ij}}{h^d}\bigl(\Tilde{\varepsilon}_i + (m^*(X_i) - \hat{m}(X_i)\bigr)\bigl(\Tilde{\varepsilon}_j + m^*(X_j) - \hat{m}(X_j)\bigr)\hat{c}_X(\hat{U}_i)\hat{c}_X(\hat{U}_j) \nonumber \\
    &= W_{n1} + 2W_{n2} + W_{n3}\,, \label{neu1}
\end{align}
with $\Tilde{\varepsilon}_i := Y_i - m^*(X_i)$ and
\begin{align*} 
    W_{n1} &:= \frac{1}{n(n-1)}\sum_{i\not= j}\frac{K_{ij}}{h^d}\Tilde{\varepsilon}_i\Tilde{\varepsilon}_j\hat{c}_X(\hat{U}_i)\hat{c}_X(\hat{U}_j)\,, 
    \\
    W_{n2} &:= \frac{1}{n(n-1)}\sum_{i\not= j}\frac{K_{ij}}{h^d}\Tilde{\varepsilon}_i\bigl(m^*(X_j) - \hat{m}(X_j)\bigr)\hat{c}_X(\hat{U}_i)\hat{c}_X(\hat{U}_j)\,,\nonumber \\
    W_{n3} &:= \frac{1}{n(n-1)}\sum_{i\not= j}\frac{K_{ij}}{h^d}\bigl(m^*(X_i) - \hat{m}(X_i)\bigr)\bigl(m^*(X_j) - \hat{m}(X_j)\bigr)\hat{c}_X(\hat{U}_i)\hat{c}_X(\hat{U}_j)\,.\nonumber 
\end{align*}
The first term $W_{n1}$ will yield the stated limit in probability, while the other two terms will turn out to be asymptotically negligible. 
For this purpose, note that \cite[Theorems 2 and 3]{Noh2013} show $m^*(x) - \hat{m}(x) = O_{\PP}(n^{-1/2})$ pointwise, but here we require stronger uniform estimates. 

\begin{lemma} \label{lemmauniform}
Define $\boldsymbol{e}^*$ by replacing $\vartheta_0$ with $\vartheta^*$ in \eqref{e(x)}$,$ and consider the estimator $\hat{\boldsymbol{e}}$ of $\boldsymbol{e}^*$ as given in \eqref{e(x)hat}$,$ and the inner copula density $c_X$ and its estimator $\hat{c}_X$. It holds that
\begin{align*}
    \|\boldsymbol{e}^* - \hat{\boldsymbol{e}}\|_\infty &= O_{\PP}\Bigl(n^{-1/2}\Bigr)\,, & 
    \|\hat{c}_X - c_X\|_\infty &= O_{\PP}\Bigl(n^{-1/2}\Bigr)\,, 
\end{align*} 
where $\boldsymbol{e}^* - \hat{\boldsymbol{e}}$ and $\hat{c}_X - c_X$ are abbreviated forms for $\boldsymbol{e}^*\bigl(\F(\cdot)\bigr) - \hat{\boldsymbol{e}}\bigl(\hat{\F}(\cdot)\bigr)$ and $c_X\bigl(\hat{\F}(\cdot), \hat{\vartheta}_n\bigr) - c_X\bigl(\F(\cdot), \vartheta_0\bigr),$ respectively.
\end{lemma}

The first term $W_{n1}$ will yield the stated limit distribution, while the other two terms will turn out to be asymptotically negligible, i.e., we prove
\begin{subequations}
\begin{align}
    \label{pd4} W_{n1} &\overset{\PP}{\longrightarrow} M^2\,, \\
    \label{det2}
    W_{n2} &= O_{\PP}(1/n)\,, \\
    \label{det3} W_{n3} &= O_{\PP}(1/n)\,. 
\end{align}
\end{subequations}

{\it Proof of \eqref{pd4}.} In a first step, we replace the estimators $\hat{c}_X(\hat{U}_i), \hat{c}_X(\hat{U}_j)$ with the true copula densities $c_X^*(U_i), c_X^*(U_j)$ evaluated in the true marginal transforms. The resulting error is negligible due to the following.
 
\begin{lemma}
\label{lemmareplace1a}
For 
\[
    \widetilde{W}_{n1} := \frac{1}{n(n-1)} \sum_{i\not=j} \frac{K_{ij}}{h^d} \Tilde{\varepsilon}_i\Tilde{\varepsilon}_j c_X^*(U_i)c_X^*(U_j)\,,
\]
it holds that $\widetilde{W}_{n1} - W_{n1} = O_{\PP}(1/n)$. 
\end{lemma}
Now, $\widetilde{W}_{n1}$ is a $U$-statistic with the symmetric kernel
\[
    H_n(Z_i, Z_j) = \frac{K_{ij}}{h^d} \Tilde{\varepsilon}_i\Tilde{\varepsilon}_j c_X^*(U_i)c_X^*(U_j)\,,
\]
where $Z_i := (X_i, Y_i)$. If $M^2 = 0,$ then we have $\Tilde{\varepsilon}_i\Tilde{\varepsilon}_j = \varepsilon_i\varepsilon_j$ and the kernel $H_n$ is degenerate, giving $\widetilde{W}_{n1} = o_{\PP}(1)$ (see also the proof of Theorem~\ref{thm3.1}). Otherwise, we compute the mean of $H_n(Z_i, Z_j)$ as follows: 
\begin{align*}
    \E(H_n(Z_i, Z_j)) &= \E(\E(H_n(Z_i, Z_j) \mid Z_i, Z_j)) \\
    &= \frac{1}{h^d} \E\Bigl(\E\bigl(K_{ij}(Y_i - m^*(X_i))(Y_j - m^*(X_j))c_X^*(U_i)c_X^*(U_j) \mid Z_i, Z_j\bigr)\Bigr) \\
    &= \frac{1}{h^d} \E\Bigl(K_{ij}(m(X_i) - m^*(X_i))(m(X_j) - m^*(X_j))c_X^*(U_i)c_X^*(U_j)\Bigr) \\
    &= \frac{1}{h^d} \iint K\left(\frac{x_i - x_j}{h}\right) (m(x_i) - m^*(x_i))(m(x_j) - m^*(x_j)) \\
    & \hspace{4cm} \times c_X^*(\F(x_i))c_X^*(\F(x_j))p(x_i)p(x_j)\dx_i\dx_j\,.
\end{align*}
By a simple substitution argument, and by use of $\int K(u)\du = 1$, one now obtains 
\begin{align*}
    \E(H_n(Z_i, Z_j)) &= \frac{1}{h^d} \iint K(u)(m(v - uh) - m^*(v - uh))(m(v) - m^*(v)) \\
    & \hspace{3cm} \times c_X^*(\F(v - uh))c_X^*(v)p(v - uh)p(v)h^d \du\dv \\
    &= \int (m(v) - m^*(v))^2 c_X^*(\F(v - uh)^2p(v)^2\dv + o_{\PP}(1) \\
    &= M^2 + o_{\PP}(1)\,.
\end{align*}
This substitution argument can also be used to verify the condition of \cite[Theorem 3.1]{powell1989semiparametric}, namely, $\E(\|H_n(Z_i, Z_j)\|^2) = O(h^{-d}) = o(n)$. This theorem gives $\widetilde{W}_{n1} = \E(\widetilde{W}_{n1}) + o_{\PP}(1) = M^2 + o_{\PP}(1)$. By Lemma~\ref{lemmareplace1a}, the same applies for $W_{n1}$. 
\medskip

{\it Proof of \eqref{det2}.} 
Using a shorthand notation omitting the arguments, we can write 
\begin{align}
    m^* - \hat{m} &= \frac{\boldsymbol{e}^*}{c_X^*} - \frac{\hat{\boldsymbol{e}}}{\hat{c}_X} = \frac{\boldsymbol{e}^*\hat{c}_X - \hat{\boldsymbol{e}}c_X^*}{c_X^*\hat{c}_X} = \frac{(\boldsymbol{e}^* - \hat{\boldsymbol{e}})c_X^* + \boldsymbol{e}^*(\hat{c}_X - c_X^*)}{c_X^*\hat{c}_X} \nonumber \\
    &= \frac{\boldsymbol{e}^* - \hat{\boldsymbol{e}}}{\hat{c}_X} + \frac{m^*(\hat{c}_X - c_X^*)}{\hat{c}_X}\,. \label{pd5}
\end{align}
Therefore, we rewrite $W_{n2}$ as
\begin{align*}
    W_{n2} &= \frac{1}{n(n-1)}\sum_{i\not=j} \frac{K_{ij}}{h^d}\Tilde{\varepsilon}_i \hat{c}_X(\hat{U}_i)\Bigl(\boldsymbol{e}^*(U_j) - \hat{\boldsymbol{e}}(\hat{U}_j)\Bigr)\, \\
    &+ \frac{1}{n(n-1)}\sum_{i\not=j} \frac{K_{ij}}{h^d}\Tilde{\varepsilon}_i \hat{c}_X(\hat{U}_i)m^*(X_j)\Bigl(\hat{c}_X(\hat{U}_j) - c_X^*(U_j)\Bigr)\,.
\end{align*}
We can replace $W_{n2}$ by 
\begin{align*}
    \widetilde{W}_{n2} 
    = E_{n1} + E_{n2}\,,
\end{align*}
only leaving an error of $O_{\PP}(1/n)$ by the arguments in the proof of Lemma~\ref{lemmauniform}, where
\begin{align*}
   E_{n1} &:= \frac{1}{n(n-1)}\sum_{i\not= j}\frac{K_{ij}}{h^d}\Tilde{\varepsilon}_i c_X^*(U_i)\Bigl(\boldsymbol{e}^*(U_j) - \hat{\boldsymbol{e}}(\hat{U}_j)\Bigr)\,, \\
   E_{n2} &:= \frac{1}{n(n-1)}\sum_{i\not=j} \frac{K_{ij}}{h^d}\Tilde{\varepsilon}_ic_X^*(U_i)m^*(X_j)\Bigl(\hat{c}_X(\hat{U}_j) - c_X^*(U_j)\Bigr)\,.
\end{align*}
The main work is to show that $E_{n1} = O_{\PP}(n^{-1})$, which is done in the following lemma. The corresponding statement $E_{n2} = O_{\PP}(n^{-1})$ is obtained from analogous arguments using the regularity assumption \ref{item1}.

\begin{lemma} \label{lemmaA1}
    We have $E_{n1} = O_{\PP}(1/n)$.
\end{lemma}

{\it Proof of \eqref{det3}.} Based on the decomposition \eqref{pd5}, we have
\begin{align*}
    W_{n3} =~& \frac{1}{n(n-1)}\sum_{i\not= j}\frac{K_{ij}}{h^d}\Bigl(m^*(X_i) - \hat{m}(X_i)\Bigr)\Bigl(m^*(X_j) - \hat{m}(X_j)\Bigr)\hat{c}_X(\hat{U}_i)\hat{c}_X(\hat{U}_j) \\
    =~& \frac{1}{n(n-1)}\sum_{i\not= j}\frac{K_{ij}}{h^d}\Bigl(\boldsymbol{e}^*(U_i) - \hat{\boldsymbol{e}}(\hat{U}_i)\Bigr)\Bigl(\boldsymbol{e}^*(U_j) - \hat{\boldsymbol{e}}(\hat{U}_j)\Bigr) \\
    &+ \frac{2}{n(n-1)}\sum_{i\not= j}\frac{K_{ij}}{h^d}\Bigl(\boldsymbol{e}^*(U_i) - \hat{\boldsymbol{e}}(\hat{U}_i)\Bigr)m^*(X_j)\Bigl(\hat{c}_X(\hat{U}_j) - c_X^*(U_j)\Bigr) \\
    &+ \frac{1}{n(n-1)}\sum_{i\not= j}\frac{K_{ij}}{h^d} m^*(X_i)m^*(X_j)\Bigl(\hat{c}_X(\hat{U}_i) - c_X^*(U_i)\Bigr)\Bigl(\hat{c}_X(\hat{U}_j) - c_X^*(U_j)\Bigr) \\
    =:~& W_{n3}^{(1)} + 2W_{n3}^{(2)} + W_{n3}^{(3)}\,,
\end{align*}
where the last equality defines $W_{n3}^{(1)}, W_{n3}^{(2)},$ and $W_{n3}^{(3)}$ in an obvious manner. For the first term $W_{n3}^{(1)}$, we have by Lemma  ~\ref{lemmauniform}:
\[
    |W_{n3}^{(1)}| \leq \frac{1}{n(n-1)} \sum_{i\not= j} \frac{|K_{ij}|}{h^d} \|\boldsymbol{e}^* - \hat{\boldsymbol{e}}\|_\infty^2 = O_{\PP}\left(\frac{1}{n}\right) \frac{1}{n(n-1)}\sum_{i\not=j} \frac{|K_{ij}|}{h^d} = O_{\PP}\left(\frac{1}{n}\right),
\]
where we used the previous substitution argument, namely,
\begin{align}
    \E\left(\frac{1}{n(n-1)}\sum_{i\not=j} \frac{|K_{ij}|}{h^d}\right) =  \E\left(\frac{|K_{12}|}{h^d}\right) &= \iint \frac{1}{h^d} \left|K\left(\frac{x_1-x_2}{h}\right)\right|\dx_1\dx_2 \nonumber \\
    &= \int |K(u)|\du = O(1)\,. \label{subst}
\end{align}
For the last term $W_{n3}^{(3)},$ we accordingly have
\begin{align*}
    |W_{n3}^{(3)}| &\leq O_{\PP}\left(\frac{1}{n}\right) \frac{1}{n(n-1)}\sum_{i\not=j} \frac{|K_{ij}| }{h^d}|m^*(X_i)m^*(X_j)|\,, \\
    \E\left[\frac{1}{n(n-1)}\sum_{i\not=j} \frac{|K_{ij}| }{h^d}|m^*(X_i)m^*(X_j)|\right]  &\leq \iint \frac{1}{h^d}\left|K\left(\frac{x_1-x_2}{h}\right)m^*(x_1)m^*(x_2)\right|p(x_1)p(x_2)\dx_1\dx_2 \\
    &= \int |K(u)||m^*(u)||m^*(u-vh)|p(u)p(v)\du\text{d}v \\
    & \leq \sqrt{\E(K_{12}^2)}\sqrt{\E(m^*(X)^4)} + o(1) = O(1)\,,
\end{align*}
where we use the regularity conditions \ref{item1} and \ref{item5}. The second term is treated in an analogous fashion, and we conclude that $W_{n3} = O_{\PP}(1/n)$. Putting \eqref{pd4}, \eqref{det2} and \eqref{det3} together, the claim follows. 

\subsection*{Proof of Theorem~\ref{thm3.1}}
If the classical null hypothesis \eqref{null1} is true, then the decomposition of $W_n$ applied in the proof of Theorem~\ref{thm2.1} looks as follows: $W_n = W_{n1} + 2W_{n2} + W_{n3},$ with
\begin{align} 
    W_{n1} &:= \frac{1}{n(n-1)}\sum_{i\not= j}\frac{K_{ij}}{h^d}\varepsilon_i\varepsilon_j\hat{c}_X(\hat{U}_i)\hat{c}_X(\hat{U}_j)\,, \label{neu2} \\
    W_{n2} &:= \frac{1}{n(n-1)}\sum_{i\not= j}\frac{K_{ij}}{h^d}\varepsilon_i\bigl(m(X_j) - \hat{m}(X_j)\bigr)\hat{c}_X(\hat{U}_i)\hat{c}_X(\hat{U}_j)\,,\nonumber \\
    W_{n3} &:= \frac{1}{n(n-1)}\sum_{i\not= j}\frac{K_{ij}}{h^d}\bigl(m(X_i) - \hat{m}(X_i)\bigr)\bigl(m(X_j) - \hat{m}(X_j)\bigr)\hat{c}_X(\hat{U}_i)\hat{c}_X(\hat{U}_j)\,.\nonumber 
\end{align}
By the arguments in the proof of Theorem~\ref{thm2.1}, we have $W_{n2} = O_{\PP}(1/n)$ and $W_{n3} = O_{\PP}(1/n)$, implying $nh^{d/2}W_{n2} = o_{\PP}(1)$ and $nh^{d/2}W_{n3} = o_{\PP}(1)$. The asymptotic normality of $nh^{d/2}W_n$ under $H_0$ follows from showing 
\[
    nh^{d/2}W_{n1} \overset{\D}{\longrightarrow} \NN(0,\sigma_0^2)\,.
\]
By Lemma~\ref{lemmareplace1a}, we can consider 
\[
    \widetilde{W}_{n1} := \frac{1}{n(n-1)} \sum_{i\not=j} \frac{K_{ij}}{h^d} \varepsilon_i\varepsilon_j c_X(U_i)c_X(U_j)\,,
\]
which can be written as a $U$-statistic $\widetilde{W}_{n1} = \displaystyle{\frac{1}{n(n-1)}\sum\nolimits_{i\not= j} H_n(Z_i, Z_j)}$ with the symmetric kernel
\[
    H_n(Z_i, Z_j) = \frac{K_{ij}}{h^d}\varepsilon_i\varepsilon_jc_X(\F(X_i))c_X(\F(X_j))\,,
\]
where $Z_i = (X_i, \varepsilon_i)$. It turns out that $\widetilde{W}_{n1}$ is a degenerate $U$-statistic, and we use Hall's theorem \cite[Theorem 1]{hall1984central} to determine the asymptotic behavior of $\widetilde{W}_{n1}$. Due to the regularity conditions \ref{item1} and \ref{item4}, the condition of Hall's theorem follows as demonstrated in \cite[Lemma 3.3a]{zheng1996consistent}. By said theorem, $n\widetilde{W}_{n1}/\sqrt{2\E\bigl(H_n(Z_1,Z_2)^2\bigr)}\overset{\D}{\longrightarrow} \NN(0,1),$ which means that the asymptotic distribution of $n\widetilde{W}_{n1}$ is determined by
\begin{align*}
    &\E(H_n(Z_1,Z_2)^2) \\
    =~& \frac{1}{h^{2d}} \int K^2\left(\frac{x_1 - x_2}{h}\right)\sigma^2(x_1)\sigma^2(x_2)c_X(\F(x_1))c_X(\F(x_2))p(x_1)p(x_2)\dx_1\dx_2 \\
    =~& \frac{1}{h^d} \int K^2(u)\sigma^2(x)\sigma^2(x-hu)c_X(\F(x))c_X(\F(x-hu))p(x)p(x-hu)\dx\du \\
    =~& \frac{1}{h^d} \int K^2(u)\du \int \sigma^2(x)^2c_X(\F(x))^2p(x)^2\dx + o(h^{-d}) \\
    =~& \frac{1}{h^d}\frac{\sigma_0^2}{2} + o(h^{-d})\,,
\end{align*}
hence, $nh^{d/2}\widetilde{W}_{n1} \overset{\D}{\longrightarrow} \NN(0,\sigma_0^2)$. By Lemma \ref{lemmareplace1a}, the same is true for $W_{n1}$.

At last, we show that the variance estimator $\hat{\sigma}_{0,n}^2$ as given in \eqref{sigmahat} is consistent. We first expand $e_i^2e_j^2$ by means of $e_i = \varepsilon_i + m(X_i) - \hat{m}(X_i)$, which gives the decomposition
\begin{align*}
    \hat{\sigma}_{0,n}^2 &= \frac{2}{n(n-1)} \sum_{i \not= j} \frac{K_{ij}^2}{h^d}\hat{c}_X(\hat{U}_i)^2\hat{c}_X(\hat{U}_j)^2  \left[\varepsilon_i^2\varepsilon_j^2 + 2\varepsilon_i^2(m(X_j) - \hat{m}(X_j))^2\right. \\
    &\hspace{6.9cm} + 4\varepsilon_i^2\varepsilon_j (m(X_j) - \hat{m}(X_j))  \\
    &\hspace{6.9cm} + 4\varepsilon_i\varepsilon_j(m(X_i) - \hat{m}(X_i))(m(X_j) - \hat{m}(X_j)) \\
    &\hspace{6.9cm} + 4\varepsilon_i (m(X_i) - \hat{m}(X_i))(m(X_i) - \hat{m}(X_j))^2 \\
    &\hspace{6.9cm} + (m(X_i) - \hat{m}(X_i))^2(m(X_j) - \hat{m}(X_j))^2\Bigr]\,.
\end{align*}
Applying \eqref{pd5} and using the abbreviation $\EE_i := \boldsymbol{e}(U_i) - \hat{\boldsymbol{e}}(\hat{U}_i) + m(X_i)\Bigl(\hat{c}_X(\hat{U}_i) - c_X(U_i)\Bigr)$, we have
\begin{align*}
    \hat{\sigma}_{0,n}^2 &= \frac{2}{n(n-1)}\sum_{i\not=j} \frac{K_{ij}^2}{h^d} \left[\hat{c}_X(\hat{U}_i)^2\hat{c}_X(\hat{U}_j)^2\varepsilon_i^2\varepsilon_j^2 + 2\hat{c}_X(\hat{U}_i)^2\varepsilon_i^2\EE_j^2 + 4\hat{c}_X(\hat{U}_i)^2\hat{c}_X(\hat{U}_j)\varepsilon_i^2\varepsilon_j\EE_j \right. \\
    & \hspace{4cm}  + 4\hat{c}_X(\hat{U}_i)\hat{c}_X(\hat{U}_j)\varepsilon_i\varepsilon_j\EE_i\EE_j + \left.4\hat{c}_X(\hat{U}_i)\varepsilon_i\EE_i\EE_j^2 + \EE_i^2\EE_j^2\right]\hspace{0.1mm}.
\end{align*}
To replace all $\hat{c}_X(\hat{U}_i), \hat{c}_X(\hat{U}_j)$ with the true quantities $c_X(U_i), c_X(U_j),$ we use a first-order Taylor expansion on the functions $(\textbf{u}_i, \textbf{u}_j, \vartheta) \mapsto c_X(\textbf{u}_i, \vartheta)^{k_1} c_X(\textbf{u}_j, \vartheta)^{k_2}$, with $k_1, k_2 \in \{0,1,2\}$. Due to $\|U_i - \hat{U}_i\|_\infty = O_{\PP}(n^{-1/2})~\forall i$ by Donsker's theorem for the empirical distribution function, and $\hat{\vartheta}_n - \vartheta_0 = O_{\PP}(n^{-1/2})$ by  \eqref{item6}, and due to the regularity assumption \ref{item4}, this leaves an error of $o_{\PP}(1),$ so we can write $\hat{\sigma}_{0,n}^2 = \widetilde{\sigma}_{0,n}^2 + o_{\PP}(1)$ with
\begin{align*}
    \widetilde{\sigma}_{0,n}^2 &:= \frac{2}{n(n-1)}\sum_{i\not=j} \frac{K_{ij}^2}{h^d} \left[c_X(U_i)^2c_X(U_j)^2\varepsilon_i^2\varepsilon_j^2 + 2c_X(U_i)^2\varepsilon_i^2\EE_j^2 + 4c_X(U_i)^2c_X(U_j)\varepsilon_i^2\varepsilon_j\EE_j \right. \\
    & \hspace{4cm}  + 4c_X(U_i)c_X(U_j)\varepsilon_i\varepsilon_j\EE_i\EE_j + \left.4c_X(U_i)\varepsilon_i\EE_i\EE_j^2 + \EE_i^2\EE_j^2\right]\hspace{0.1mm}.
\end{align*}
We decompose $\widetilde{\sigma}_{0,n}^2$ further into
\[
    \widetilde{\sigma}_{0,n}^2 = S_{n1} + S_{n2}\,,
\]
where 
\begin{align*} 
    S_{n1} &:= \frac{2}{n(n-1)} \sum_{i \not= j} \frac{K_{ij}^2}{h^d} c_X(U_i)^2c_X(U_j)^2\varepsilon_i^2\varepsilon_j^2\,, \\
    S_{n2} &:= \frac{2}{n(n-1)}\sum_{i\not=j} \frac{K_{ij}^2}{h^d} \left[2c_X(U_i)^2\varepsilon_i^2\EE_j^2 + 4c_X(U_i)^2c_X(U_j)\varepsilon_i^2\varepsilon_j\EE_j \right. \\
    & \hspace{4cm}  + 4c_X(U_i)c_X(U_j)\varepsilon_i\varepsilon_j\EE_i\EE_j + \left.4c_X(U_i)\varepsilon_i\EE_i\EE_j^2 + \EE_i^2\EE_j^2\right]\hspace{0.1mm}. 
\end{align*}
By the same arguments as given in the proof of \cite[Lemma 3.3e]{zheng1996consistent}, and by the boundedness of $c_X$, we have $S_{n1} \overset{{\PP}}{\longrightarrow} \sigma_0^2,$ while in $S_{n2}$, each summand involves at least one factor $\EE_j$, which is uniformly of order $O_{\PP}(n^{-1/2})$ by  Lemma~\ref{lemmauniform}. So, by the substitution argument as in \eqref{subst}, we have $S_{n2} = o_{\PP}(1)$. In conclusion, $\hat{\sigma}_{0,n}^2 \overset{\PP}{\longrightarrow} \sigma_0^2$, which completes the proof of Theorem~\ref{thm3.1}. \hfill \qed

\subsection*{Proof of Theorem~\ref{thm4.1}}

We now decompose $W_n$ in a way that addresses both the estimation error $\hat{m} - m^*$ and the model misspecification $m - m^*$. Take $W_{n1}$ as in \eqref{neu2} and set 
\begin{align*}
    W_{n2}^{(1)} &:= \frac{1}{n(n-1)}\sum_{i\not=j} \frac{K_{ij}}{h^d} \varepsilon_i \bigl(\hat{m}(X_j) - m^*(X_j)\bigr)\hat{c}_X(\hat{U}_i)\hat{c}_X(\hat{U}_j)\,, \\
    W_{n2}^{(2)} &:= \frac{1}{n(n-1)}\sum_{i\not=j} \frac{K_{ij}}{h^d} \varepsilon_i \Delta(X_j)\hat{c}_X(\hat{U}_i)\hat{c}_X(\hat{U}_j)\,, \\
    W_{n3}^{(1)} &:= \frac{1}{n(n-1)}\sum_{i\not=j} \frac{K_{ij}}{h^d}\bigl(\hat{m}(X_i) - m^*(X_i)\bigr)\bigl(\hat{m}(X_j) - m^*(X_j)\bigr)\hat{c}_X(\hat{U}_i)\hat{c}_X(\hat{U}_j)\,, \\
    W_{n3}^{(2)} &:= \frac{1}{n(n-1)}\sum_{i\not=j} \frac{K_{ij}}{h^d}\Delta(X_i)\bigl(\hat{m}(X_j) - m^*(X_j)\bigr)\hat{c}_X(\hat{U}_i)\hat{c}_X(\hat{U}_j)\,, \\
    W_{n3}^{(3)} &:= \frac{1}{n(n-1)}\sum_{i\not=j} \frac{K_{ij}}{h^d} \Delta(X_i)\Delta(X_j)\hat{c}_X(\hat{U}_i)\hat{c}_X(\hat{U}_j)\,,
\end{align*}
hence, we obtain the decomposition 
\[
W_n = W_{n1} - 2\Bigl(W_{n2}^{(1)} - W_{n2}^{(2)}\Bigr) + \Bigl(W_{n3}^{(1)} - 2W_{n3}^{(2)} + W_{n3}^{(3)}\Bigr)\,.
\]
Reiterating the arguments in the proofs of Theorem~\ref{thm2.1} and~\ref{thm3.1}, we have that 
\begin{align*}
    W_{n1}  & = O_{\PP}\left((nh^{d/2})^{-1}\right) 
    = O_{\PP}\left(n^{-1/2} \cdot (nh^d)^{-1/2}\right) 
    = o_{\PP}(n^{-1/2}) \,, \\
 W_{n2}^{(1)} &= o_{\PP}\left((nh^{d/2})^{-1}\right)  
    = o_{\PP}(n^{-1/2}) \,,\qquad W_{n3}^{(1)} = o_{\PP}(n^{-1/2})\,,
\end{align*}
so these terms do not contribute to the asymptotic behavior of $\sqrt{n}W_n$, i.e.,
\[
    \sqrt{n}W_n = 2\sqrt{n}\left(W_{n2}^{(2)} - W_{n3}^{(2)}\right) + \sqrt{n}W_{n3}^{(3)} + O_{\PP}\left(\frac{1}{\sqrt{n}h^{d/2}}\right). 
\]
As in the proof of Theorem~\ref{thm3.1}, we need to replace $\hat{c}_X(\hat{U}_i)\hat{c}_X(\hat{U}_j)$ with its true analogue, which is now $c_X^*(U_i)c_X^*(U_j)$ for $c_X^*(\textbf{u}) = c_X(\textbf{u}, \vartheta^*)$. By the same arguments as before (using Lemma~\ref{lemmauniform} for $W_{n3}^{(2)}$ and the proof mechanism of Lemma~\ref{lemmareplace1a} for $W_{n2}^{(2)}$ and $W_{n3}^{(3)}$), this only leaves an error of $O_{\PP}(1/n)$. We obtain 
\begin{equation}
    \sqrt{n}W_n = 2\sqrt{n}\left(\widetilde{W}_{n2}^{(2)} - \widetilde{W}_{n3}^{(2)}\right) + \sqrt{n}\widetilde{W}_{n3}^{(3)} + O_{\PP}\left(\frac{1}{\sqrt{n}}\right), \label{eqn5}
\end{equation}
where $\widetilde{W}_{n2}^{(2)}, \widetilde{W}_{n2}^{(3)}, \widetilde{W}_{n3}^{(3)}$ emerge from $W_{n2}^{(2)}, W_{n2}^{(3)}, W_{n3}^{(3)}$, respectively, by replacing $\hat{c}_X(\hat{U}_i)\hat{c}_X(\hat{U}_j)$ with $c_X^*(U_i)c_X^*(U_j)$. 

To analyze $\sqrt{n}\left(\widetilde{W}_{n3}^{(2)} - \widetilde{W}_{n2}^{(2)}\right)$, we use the asymptotic representation in \cite[Theorem 3]{Noh2013}, and, upon expanding $\partial \boldsymbol{e}^* = \partial (m^*c_X^*) = c_X^*\partial m^* + m^* \partial c_X^*$, we obtain
\[
    \sqrt{n}(\hat{m}(x) - m^*(x)) = \frac{1}{\sqrt{n}}\frac{1}{c_X^*(\F(x))}\sum_{k=1}^n \M_k(x) + o_{\PP}(1)
\]
with $\M_k(x)$ as stated in \eqref{mkx}. Therefore, 
\begin{align*}
    \sqrt{n}\left(\widetilde{W}_{n3}^{(2)} - \widetilde{W}_{n2}^{(2)}\right) &= \frac{\sqrt{n}}{n(n-1)} \sum_{i\not=j} \frac{K_{ij}}{h^d} \Delta^*(X_i) c_X^*(U_j)(\hat{m}(X_j) - m^*(X_j) - \varepsilon_j)  \\
    &= \frac{\sqrt{n}}{n(n-1)}\sum_{i\not=j} \frac{K_{ij}}{h^d} \Delta^*(X_i)\frac{1}{n}\sum_{k=1}^n [\M_k(X_j) - \varepsilon_j] + o_{\PP}(1)\,. 
\end{align*}
We first ignore the colliding indices $k=i$ and $k=j,$ i.e., we disregard 
\[
    \frac{\sqrt{n}}{n^2(n-1)}\sum_{i\not=j} \frac{K_{ij}}{h^d} \Delta^*(X_i)(\M_i(X_j) - \varepsilon_i + \M_j(X_j) - \varepsilon_j) = o_{\PP}(1)\,,
\]
which is assured by regularity. Then, we regroup the summation indices so that 
\begin{align*}  \sqrt{n}\left(\widetilde{W}_{n3}^{(2)} - \widetilde{W}_{n2}^{(2)}\right) &= \sqrt{n}\,W_{n4} + o_{\PP}(1)\,,
\end{align*}
where 
\begin{align}
\label{det100}
W_{n4} &:= \frac{1}{n(n-1)}\sum_{i\not=j} (\M_i(X_j) - \varepsilon_j)\frac{1}{n} \sum_{\substack{k=1 \\ k\not=i,j}}^n \frac{K_{kj}}{h^d}\Delta^*(X_k)\,.
\end{align}

\begin{lemma} \label{lemma4.2}
The random variable 
\[
    W_{n4} = \frac{1}{n^2(n-1)}\sum_{i\neq j\neq k} \frac{K_{ij}}{h^d}\Delta^*(X_i)(\M_k(X_j) - \varepsilon_j)  
\]
is centered. 
\end{lemma}

Next, we replace the kernel estimation $n^{-1}\displaystyle{\sum\nolimits_{k=1, k\not=j}^n h^{-d}K_{kj}\Delta^*(X_k)}$ by the actual density $\Delta^*(X_j)p(X_j)$ to obtain
\begin{equation}
    \overline{W}_{n4} := \frac{1}{n(n-1)}\sum_{i\not=j} (\M_i(X_j) - \varepsilon_j) \Delta^*(X_j)p(X_j)\,, \label{wn4quer}
\end{equation}
which is centered for the same reasons as in Lemma \ref{lemma4.2}. We can show that $\sqrt{n}(W_{n4} - \overline{W}_{n4}) \rightarrow 0$ in quadratic mean. 

\begin{lemma} \label{lemma4.3}
    $$
    \E\left((W_{n4} - \overline{W}_{n4})^2\right) = o(n^{-1})\,.
    $$
\end{lemma}

Overall, we have so far demonstrated that
\begin{align}
    \sqrt{n}(W_n - \mu_n) &= \sqrt{n}\left(\widetilde{W}_{n3}^{(3)} - \mu_n + 2\left(\widetilde{W}_{n3}^{(2)} - \widetilde{W}_{n2}^{(2)}\right)\right) + o_{\PP}(1) \label{decomp} \\
    &= \sqrt{n}\left(\widetilde{W}_{n3}^{(3)} - \mu_n + 2\overline{W}_{n4}\right) + o_{\PP}(1)\,. \nonumber
\end{align}
Noting that $\E\left(\widetilde{W}_{n3}^{(3)} + 2\overline{W}_{n4}\right) = \E\left(\widetilde{W}_{n3}^{(3)}\right) = \mu_n,$ we can represent $\widetilde{W}_{n3}^{(3)} + 2\overline{W}_{n4}$ as a $U$-statistic with the symmetric kernel
\[
    \HH_n(Z_i, Z_j) := H_n(Z_i, Z_j) + 2\HH(Z_i, Z_j)\,,
\]
where 
\begin{align}
    H_n(Z_i, Z_j) &:= h^{-d}K_{ij}\Delta^*(X_i)\Delta^*(X_j)\,, \label{hnxixj} \\
    \HH(Z_i,Z_j) &:= \frac{1}{2}\Bigl((\M_i(X_j) - \varepsilon_j)\Delta^*(X_j)p(X_j) + (\M_j(X_i) - \varepsilon_i)\Delta^*(X_i)p(X_i)\Bigr)\,. \label{hzizj}
\end{align}
The asymptotic normality of 
\[
    \sqrt{n}\left(\widetilde{W}_{n3}^{(3)} + 2\overline{W}_{n4} - \mu_n \right) = \sqrt{n}\left(\frac{1}{n(n-1)}\sum_{i \neq j} \HH_n(Z_i, Z_j) - \mu_n\right)
\]
is clear from standard arguments (see, e.g., \cite[Theorem 2]{dette2004some}). To determine the limiting variance, we first consider the two parts $\sqrt{n}\left(W_{n3}^{(3)} - \mu_n\right)$ and $\sqrt{n}\,\overline{W}_{n4}$ individually. 
The variance of $\sqrt{n}\,\overline{W}_{n4}$ is explicitly calculated as follows:

\begin{lemma} \label{lemma4.4}
    The variance of $\sqrt{n}\,\overline{W}_{n4}$ is given by
\begin{align*}
    \Var(\sqrt{n}\,\overline{W}_{n4}) &= \frac{n-2}{n-1}\Bigl(\E\bigl(\M_1(X_2)\M_1(X_3)\Delta^*(X_2)\Delta^*(X_3)p(X_2)p(X_3)\bigr) \\
    & \hspace{1cm}+ \E\bigl(\varepsilon_1^2 \Delta^*(X_1)^2 p(X_1)^2\bigr)\Bigr) + o_{\PP}(1)\,.
\end{align*}
\end{lemma}

It immediately follows that 
\[
    \lim_{n\rightarrow\infty} \Var(\sqrt{n}\,\overline{W}_{n4}) = \sigma_2^2
\]
with $\sigma_2^2$ as stated in the claim. By \eqref{decomp}, this also yields
\begin{equation}
    \lim_{n\rightarrow\infty} \Var\left(2\sqrt{n}\left(\widetilde{W}_{n3}^{(2)} - \widetilde{W}_{n2}^{(2)}\right)\right) = 4\sigma_2^2\,. \label{sigma2}
\end{equation}

Next, note that $\widetilde{W}_{n3}^{(3)}$ itself is a $U$-statistic with the symmetric kernel $H_n(X_i, X_j) = H_n(Z_i, Z_j)$ as in \eqref{hnxixj}. Under the alternative hypothesis, this kernel is non-degenerate. Due to \cite[Theorem 3.1]{powell1989semiparametric}, we can use the H\'{a}jek projection $g_n(X_1) := \E(H_n(X_1, X_2) \mid X_1)$ to find the limiting variance of $\sqrt{n}\,\widetilde{W}_{n3}^{(3)}$, that is,
\[
    \Var\left(\sqrt{n}\,\widetilde{W}_{n3}^{(3)}\right) = 4\Var(g_n(X_1)) + o(1)\,.
\]
We have by standard calculations:
\begin{align*}
    \Var\left(g_n(X_1)\right) &= \Cov\left(H_n(X_1, X_2), H_n(X_1, X_3)\right) \\
    &= \E\bigl(H_n(X_1, X_2)H_n(X_1, X_3)\bigr) - \E\bigl(H_n(X_1, X_2)\bigr)^2\,.
\end{align*}
By the substitution argument, we have
\begin{align*}
    &\E\bigl(H_n(X_1, X_2)H_n(X_1, X_3)\bigr) \\
    =~& \iiint \frac{1}{h^{2d}}K\left(\frac{u-v}{h}\right)K\left(\frac{u-w}{h}\right)\Delta^*(u)^2\Delta^*(v)\Delta^*(w)p(u)p(v)p(w)\,\du\,\dv\,\dw \\
    =~& \iiint K(u)^2 \Delta^*(v)^2 \Delta^*(v+hu) \Delta^*(v+hw)p(v)p(v+hu)p(v+hw)\,\du\,\dv\,\dw \\
    =~& \int K(u)^2\du \cdot \int \Delta^*(v)^4p(v)^3\dv + o(1) = \E\Bigl(\bigl(\Delta^*(X_1)^2p(X_1)\bigr)^2\Bigr) + o(1)\,,
\end{align*}
and likewise,
\begin{align*}
    \E\bigl(H_n(X_1, X_2)\bigr)^2 &= \left(\iint \frac{1}{h^d} K(u) \Delta^*(v)\Delta^*(v+hu)p(u)p(v+hu)\,\du\,\dv\right)^2 \\
    &= \E\left(\Delta^*(X_1)^2p(X_1)\right)^2 + o(1)\,.
\end{align*}
In conclusion, $\Var\left(g_n(X_1)\right) = \Var(\Delta^*(X_1)^2p(X_1)) + o(1)$ and 
\begin{equation}
    \lim_{n\rightarrow\infty} \Var\left(\sqrt{n}\,\widetilde{W}_{n3}^{(3)}\right) = 4\Var(\Delta^*(X_1)^2p(X_1)) = \sigma_1^2\,. \label{sigma1}
\end{equation}
To conclude the proof, we require the following:

\begin{lemma} \label{lemma4.5}
$\sqrt{n}\left(\widetilde{W}_{n3}^{(2)} - \widetilde{W}_{n2}^{(2)}\right)$ and $\sqrt{n}\,\widetilde{W}_{n3}^{(3)}$ are asymptotically uncorrelated. 
\end{lemma}

Putting \eqref{decomp}, \eqref{sigma2} and \eqref{sigma1} together, we obtain
\[
    \lim_{n\rightarrow\infty} \Var(\sqrt{n}(W_n - \mu_n)) = 4(\sigma_1^2 + \sigma_2^2)\,,
\]
from which the claim follows. \hfill \qed 

\subsection*{Proof of Theorem~\ref{thm4.6}}

We begin with the following important observation on convergence of $U$-processes.

\begin{lemma} \label{lemma5.2} 
\emph{a)} Let 
$
U_n = (n(n-1))^{-1}\sum_{i\neq j} h(X_i, X_j)
$ 
be a $U$-statistic with a fixed symmetric kernel $h(x, y)$ of order $2$ which is non-degenerate, that is, $\sigma^2 := 4\Var(\E(h(X_1, X_2) \mid X_1))>0 $.
Then, for the $U$-process 
$$
U_{\nt} := { 1\over n(n-1)}\sum_{i\neq j}^{\nt} h(X_i, X_j)\,,
$$
it holds that 
\[
    \{\sqrt{n}(U_{\nt} - t^2r)\}_{t \in [0,1]} \overset{\D}{\longrightarrow} \{\sigma\hspace{0.2mm} t B(t)\}_{t \in [0,1]}\,,
\] 
where $r := \E(U_n)$ and $B$ is a standard Brownian motion.
\par \medskip 
\emph{b)} Let $U_n = (n(n-1))^{-1}\sum_{i\neq j} h_n(X_i, X_j)$ be a non-degenerate $U$-statistic with size-dependent symmetric kernels $h_n(x, y)$. Take $r_n := \E(U_n),$ $g_n(X_1) := \E(h_n(X_1, X_2) \mid X_1),$ and $\sigma_{n1}^2 := 4\Var(g_n(X_1))$. If the kernels satisfy the following assumptions:
\begin{itemize}
\item $\E(h_n(X_1, X_2)^2) = o(n),$
\item the fourth central moments $\E((g_n(X_1) - r_n)^4)$ are uniformly bounded,
\item $\sigma_{n1}^2 \longrightarrow \sigma^2$ for some $\sigma^2 > 0,$
\end{itemize}
then the $U$-process
$$
U_{\nt} := {1\over n(n-1)}\sum_{i\neq j}^{\nt} h_n(X_i, X_j)
$$ 
satisfies
\[
    \{\sqrt{n}(U_{\nt} - t^2r_n)\}_{t \in [0,1]} \overset{\D}{\longrightarrow} \{\sigma \hspace{0.2mm} t B(t)\}_{t \in [0,1]}
\] 
with a standard Brownian motion $B$.
\end{lemma}

Note that the assumptions of Lemma~\ref{lemma5.2}b) are satisfied by $U$-statistics with kernels of the form
\[
    h_n(X_i, X_j) = \frac{K_{ij}}{h^d}f(X_i, X_j) 
\]
for any real-valued function $f(X_i, X_j)$ that is sufficiently regular in the sense of our assumptions \ref{item1} and \eqref{item6}. In this case,
\[
    g_n(X_1) = \int h^{-d}K\left(\frac{X_1 - y}{h}\right)f(X_i, y)p(y)\dy\,.
\]
The three assumptions of Lemma~\ref{lemma5.2}b) all follow from substitution arguments. First,
\begin{align*}
    \E(h_n(X_1, X_2)^2) &= \iint \frac{1}{h^{2d}}K\left(\frac{x-y}{h}\right)^2f(x,y)^2p(x)p(y)\dx\dy \\
    &= \frac{1}{h^d} \iint K(u)^2f(x, x-hu)^2p(x)p(x-hu)\du\dx = O(h^{-d}) = o(n)\,,
\end{align*}
and $\sigma_{n1}^2 = 4\E(g_n(X_1)^2) - 4r_n^2,$ where 
\begin{align*}
    r_n &= \iint \frac{1}{h^d} K\left(\frac{x-y}{h}\right)f(x,y)p(x)p(y)\dx\dy \\
    &= \iint K(u)f(x-hu,x)p(x-hu)p(x)\du\dx \\
    &\longrightarrow \int f(x,x)p(x)^2\dx = \E(f(X,X)p(X))\,,
\end{align*}
and 
\begin{align*}
    \E(g_n(X_1)^2) &= \int \left(\int K(u)f(x, x-hu)p(x-hu)\du\right)^2p(x)\dx \\
    &\longrightarrow \int f(x,x)^2p(x)^3\dx = \E(f(X,X)^2p(X)^2)\,,
\end{align*}
which implies $\sigma_{n1}^2 \longrightarrow \sigma^2$ with $\sigma^2 := 4\Var(f(X,X)p(X))$. Likewise, we see that
\[
    \E(g_n(X_i)^4) \longrightarrow \int f(x,x)^4p(x)^5\dx = \E(f(X,X)^4p(X)^4) < \infty\,,
\]
from which the uniform boundedness of fourth central moments follows as well.

We now perform the steps leading to \eqref{eqn5} and the subsequent H\'{a}jek approximations for the sequential process.

\begin{lemma} \label{lemma4.7}
For the process  $\{ W_{\nt}\}_{t \in [\delta,1]}$ as defined in \eqref{wnt}, 
we have 
\[
    \max_{t \in [0,1]} \sqrt{n}\left(W_{\nt} - t\mu_n - \frac{2}{n}\sum_{i=1}^{\nt} (g_n(X_i) + 2g(Z_i) - \mu_n)\right) = o_{\PP}(1)\,,
\]
where $g_n(X_i) = \E(H_n(X_i, X_j) \mid X_i)$ and $g(Z_i) := \E(\HH(Z_i, Z_j) \mid Z_i)$ with $H_n$ and $\HH$ defined in \eqref{hnxixj} and \eqref{hzizj}\emph{,} respectively. 
\end{lemma}

From arguments in the proof of Lemma~\ref{lemma5.2}b), it follows that 
\[
    \left\{\frac{2}{\sqrt{n}}\sum_{i=1}^{\nt} (g_{\nt}(X_i) + 2g(Z_i) - \mu_{\nt})\right\}_{t \in [\delta,1]} \overset{\D}{\longrightarrow} \{\sigma B(t)\}_{t \in [\delta,1]}\,,
\]
where $\sigma^2$ is the limiting variance of the original $U$-statistic $\sqrt{n}(W_n - \mu_n),$ which is given by Theorem~\ref{thm4.1}. In combination with Lemma \ref{lemma4.7}, we have
\[
    \{\sqrt{n}(W_{\nt} - t\mu_n)\}_{t \in [\delta,1]} \overset{\D}{\longrightarrow} \{\sigma B(t)\}_{t \in [\delta,1]}\,.
\]
From the continuous mapping theorem, we obtain
\begin{align*}
    \frac{W_n - \mu_n}{\int |W_{\nt} - tW_n|\dt} =~& \frac{\sqrt{n}(W_n - \mu_n)}{\sqrt{n}\int |(W_{\nt} - t\mu_n) - t(W_n - \mu_n)|\dt} \\
    \overset{\D}{\longrightarrow}~& \frac{\sigma B(1)}{\sigma \int |B(t) - tB(1)|\dt} = W\,,
\end{align*}
from which the claim immediately follows. \hfill \qed 

\newpage 

\section{Proofs of Auxiliary Lemmas} \label{appendix_b}
 \def\theequation{B.\arabic{equation}}	
   \setcounter{equation}{0}

\subsection*{Proof of Lemma~\ref{lemmauniform}}
Let $x \in \R^d$. Following the proof of \cite[Theorem 1]{Noh2013}, we first approximate $\boldsymbol{e}^*(x)$ by 
\[
    \overline{\boldsymbol{e}}(x) := n^{-1}\sum_{i=1}^n Y_i c(F_0(Y_i), \F(x), \vartheta^*)\,,
\]
and analyze both $\|\boldsymbol{e}^*(x) - \overline{\boldsymbol{e}}(x)\|_\infty$ and $\|\hat{\boldsymbol{e}}(x) - \overline{\boldsymbol{e}}(x)\|_\infty$ separately. For the first term, we interpret $\boldsymbol{e}^*$, $\overline{\boldsymbol{e}}$ as functions of the random variables $Y_1, \ldots, Y_n$, i.e., for given $\mathbf{u} = \F(x) \in [0,1]^d$, we put
\[
    f_\textbf{u}(Y) := Yc(F_0(Y), \textbf{u}, \vartheta^*)\,,
\]
giving $\overline{\boldsymbol{e}}(x) = n^{-1}\sum_{k=1}^n f_{\textbf{u}}(Y_k)$ and $\boldsymbol{e}^*(x) = \E(f_{\textbf{u}}(Y))$. Take $\mathfrak{F} := \{f_\textbf{u} \mid \textbf{u} \in [0,1]^d\}$ as a class of functions parametrized on the compact space $[0,1]^d$. Due to the regularity assumptions, $c$ is Lipschitz continuous in $\textbf{u}$, i.e., for $\textbf{u}_1, \textbf{u}_2 \in [0,1]^d,$ we have
\[
    |f_{\textbf{u}_1}(y) - f_{\textbf{u}_2}(y)| = y\Bigl(c(F_0(y), \textbf{u}_1, \vartheta^*) - c(F_0(y), \textbf{u}_2, \vartheta^*)\Bigr) \leq yL\|\textbf{u}_1 - \textbf{u}_2\|
\]
for some constant $L > 0$. By \cite[Example 19.7]{van2000asymptotic}, the class $\mathfrak{F}$ is Donsker and we conclude that 
\[
    \|\boldsymbol{e}^*(x) - \overline{\boldsymbol{e}}(x)\|_\infty = O_{\PP}(n^{-1/2})\,.
\]
For the second term $\|\hat{\boldsymbol{e}}(x) - \overline{\boldsymbol{e}}(x)\|_\infty$, by Taylor expansion, we write 
\[
    \|\hat{\boldsymbol{e}}(x) - \overline{\boldsymbol{e}}(x)\|_\infty \leq \|V_{n1}(x)\|_\infty + \|V_{n2}(x)\|_\infty + \|V_{n3}(x)\|_\infty\,,
\]
where
\begin{align*}
    V_{n1}(x) &:= n^{-1}\sum_{k=1}^n Y_k\Bigl(\hat{F}_0(Y_k) - F_0(Y_k)\Bigr)\partial_{u_0}c(F_0(Y_k), \F(x), \vartheta^*) + R_{n1}(x)\,, \\
    V_{n2}(x) &:= n^{-1}\sum_{k=1}^n Y_k\Bigl(\hat{\F}(x) - \F(x)\Bigr)^\top \partial_{\textbf{u}}c(F_0(Y_k), \F(x), \vartheta^*) + R_{n2}(x)\,, \\
    V_{n3}(x) &:= n^{-1}\sum_{k=1}^n Y_k(\hat{\vartheta}_n - \vartheta^*)^\top \partial_{\vartheta} c(F_0(Y_k), \F(x), \vartheta^*) + R_{n3}(x)
\end{align*}
with remainder terms
\begin{align*}
    R_{n1}(x) &:= n^{-1}\sum_{k=1}^n Y_k\Bigl(\hat{F}_0(Y_k) - F_0(Y_k)\Bigr)\Bigl(\partial_{u_0} c(\widetilde{U}_{0k}, \widetilde{U}_{k}, \widetilde{\vartheta}_{k}) - \partial_{u_0} c(F_0(Y_k), \F(x), \vartheta^*)\Bigr)\,, \\
    R_{n2}(x) &:= n^{-1}\Bigl(\hat{\F}(x) - \F(x)\Bigr)^\top \sum_{k=1}^n Y_k \Bigl(\partial_{\textbf{u}} c(\widetilde{U}_{0k}, \widetilde{U}_{k}, \widetilde{\vartheta}_{k}) - \partial_{\textbf{u}} c(F_0(Y_k), \F(x), \vartheta^*)\Bigr)\,, \\
    R_{n3}(x) &:= n^{-1}(\hat{\vartheta}_n - \vartheta^*)^\top\sum_{k=1}^n Y_k\left(\partial_{\vartheta}c(\widetilde{U}_{0k}, \widetilde{U}_{k}, \widetilde{\vartheta}_{k}) - \partial_{\vartheta} c(F_0(Y_k), \F(x), \vartheta^*)\right)
\end{align*}
for intermediate spots $\widetilde{U}_{0k} = F_0(Y_k) \pm t_{nk}(\hat{F}_0(Y_k) - F_0(Y_k)),$ $\widetilde{U}_{k} = \F(x) \pm t_{nk}(\hat{\F}(x) - \F(x)),$ $\widetilde{\vartheta}_{k} = \vartheta^* \pm t_{nk}(\hat{\vartheta}_n - \vartheta^*)$ with some random $t_{nk} \in [0,1]$. 

As noted in the proof of \cite[Lemma 1]{Noh2013}, due to Donsker's theorem, we have $\|\hat{F}_0 - F_0\|_\infty = O_{\PP}(n^{-1/2})$ and $\|\hat{\F} - \F\|_\infty = O_{\PP}(n^{-1/2}).$ We also know that $\hat{\vartheta}_n - \vartheta^* = O_{\PP}(n^{-1/2})$ by  \eqref{item6}. The remaining estimates such as $\partial_{u_0} c(\widetilde{U}_{0k}, \widetilde{U}_{k}, \widetilde{\vartheta}_{k}) - \partial_{u_0} c(F_0(Y_k), \F(x), \vartheta^*) = o_{\PP}(1)$ are also uniform in the second component, due to the regularity condition \ref{item3}. This shows 
\[
    \|\boldsymbol{e}^*(x) - \hat{\boldsymbol{e}}(x)\|_\infty \leq \|\boldsymbol{e}^*(x) - \overline{\boldsymbol{e}}(x)\|_\infty + \|\overline{\boldsymbol{e}}(x) - \hat{\boldsymbol{e}}(x)\|_\infty = O_{\PP}(n^{-1/2})\,.
\]
For $\|\hat{c}_X - c_X\|_\infty,$ we directly employ a Taylor expansion which gives analogous uniform error bounds. \hfill \qed

\subsection*{Proof of Lemma~\ref{lemmareplace1a}}
We have to consider
\begin{align*}
    \widetilde{W}_{n1} - W_{n1} &= \frac{1}{n(n-1)}\sum_{i\not=j} \frac{K_{ij}}{h^d} \Tilde{\varepsilon}_i\Tilde{\varepsilon}_j \Bigl(c_X(\hat{U}_i, \hat{\vartheta}_n)c_X(\hat{U}_j, \hat{\vartheta}_n) - c_X(U_i, \vartheta^*)c_X(U_j, \vartheta^*)\Bigr) \\
    &= \frac{1}{n(n-1)}\sum_{i\not=j} \frac{K_{ij}}{h^d} \Tilde{\varepsilon}_i\Tilde{\varepsilon}_j \Bigl( g_X(\hat{U}_i, \hat{U}_j, \hat{\vartheta}_n) - g_X(U_i, U_j, \vartheta^*)\Bigr)
\end{align*}
with $g_X(\textbf{u}_i, \textbf{u}_j, \vartheta) := c_X(\textbf{u}_i, \vartheta)c_X(\textbf{u}_j, \vartheta)$. A second-order Taylor expansion in all three arguments gives
\begin{align}
    &g_X(\hat{U}_i, \hat{U}_j, \hat{\vartheta}_n) - g_X(U_i, U_j, \vartheta^*) \nonumber \\
    =~& (\hat{U}_i - U_i) \partial_{\textbf{u}} g(U_i, U_j, \vartheta^*) + (\hat{U}_j - U_j)\partial_{\textbf{u}} g(U_i, U_j, \vartheta^*) + (\hat{\vartheta}_n - \vartheta^*) \partial_\vartheta g(U_i, U_j, \vartheta^*) \nonumber \\
    &\quad + \frac{1}{2}\begin{pmatrix}  \hat{U}_i - U_i \\ \hat{U}_j - U_j \\ \hat{\vartheta}_n - \vartheta^*\end{pmatrix}^\top \mathrm{H}g_X(\widetilde{U}_{i}, \widetilde{U}_{j}, \widetilde{\vartheta})\begin{pmatrix}  \hat{U}_i - U_i \\ \hat{U}_j - U_j \\ \hat{\vartheta}_n - \vartheta^*\end{pmatrix} \nonumber \\
    =~& (\hat{U}_i - U_i)c_X(U_j,\vartheta^*)\partial_{\textbf{u}} c_X(U_i, \vartheta^*) + (\hat{U}_j - U_j) c_X(U_i,\vartheta^*)\partial_{\textbf{u}} c_X(U_j, \vartheta^*) \nonumber \\
    & \quad + (\hat{\vartheta}_n - \vartheta^*)\Bigl(c_X(U_j,\vartheta^*)\partial_{\vartheta} c_X(U_i, \vartheta^*) + c_X(U_i,\vartheta^*)\partial_{\vartheta} c_X(U_j, \vartheta^*)\Bigr) \nonumber \\
    &\quad + \frac{1}{2}\begin{pmatrix}  \hat{U}_i - U_i \\ \hat{U}_j - U_j \\ \hat{\vartheta}_n - \vartheta^*\end{pmatrix}^\top \mathrm{H}g_X(\widetilde{U}_{i}, \widetilde{U}_{j}, \widetilde{\vartheta})\begin{pmatrix}  \hat{U}_i - U_i \\ \hat{U}_j - U_j \\ \hat{\vartheta}_n - \vartheta^*\end{pmatrix}, \label{taylor}
\end{align}
where $\widetilde{U}_i, \widetilde{U}_j, \widetilde{\vartheta}$ are intermediate spots between $U_i, U_j, \vartheta^*$ and $\hat{U}_i, \hat{U_j}, \hat{\vartheta}_n$, respectively. The Hessian matrix $\mathrm{H}g_X(\widetilde{U}_{i}, \widetilde{U}_{j}, \widetilde{\vartheta})$ contains expressions of the type 
\begin{align*}
    &c_X(U_i, \vartheta) \partial^2_{z_1z_2} c_X(\textbf{u}_j, \vartheta)\,, & &\partial_{z_1} c_X(\textbf{u}_i, \vartheta)\partial_{z_2} c_X(U_j, \vartheta)\,,
\end{align*}
with $z_1, z_2 \in \{\textbf{u}, \vartheta\}$. The quadratic forms each involve two terms of order $O_{\PP}(1/\sqrt{n})$, i.e., we only need to ensure that, e.g., 
\[
    \frac{1}{2n(n-1)} \sum_{i\not=j} \frac{K_{ij}}{h^d} \Tilde{\varepsilon}_i\Tilde{\varepsilon}_j c_X(\widetilde{U}_j, \widetilde{\vartheta})\partial_{\textbf{u}^2}^2c_X(\widetilde{U}_i, \widetilde{\vartheta}) = O_{\PP}(1)\,,
\]
which holds due to the regularity condition \ref{item4} and can be easily seen from the substitution argument as in \eqref{subst}.

For the first-order terms, we consider the exemplary term 
\[
    \mathcal{U}_n := \frac{1}{n(n-1)}\sum_{i\not=j} \frac{K_{ij}}{h^d} \Tilde{\varepsilon}_i\Tilde{\varepsilon}_j c_X(U_j, \vartheta^*)\partial_\textbf{u} c_X(U_i, \vartheta^*) 
\]
and show that it is of order $O_{\PP}(n^{-1/2})$. We can rewrite it as a $U$-statistic with the symmetric and mean-zero kernel
\[
    H_n(Z_i, Z_j) = \frac{1}{2}\frac{K_{ij}}{h^d}\Tilde{\varepsilon}_i\Tilde{\varepsilon}_j\bigl(c_X(U_i, \vartheta^*)\partial_\textbf{u}c_X(U_j, \vartheta^*) + c_X(U_j, \vartheta^*) \partial_\textbf{u} c_X(U_i, \vartheta^*)\bigr)\,.
\]
By similar arguments as in the proof of \cite[Lemma 3.3b]{zheng1996consistent}, we obtain that $\E(\|H_n(Z_i, \linebreak Z_j)\|^2) = O(h^{-d}) = o(n)$. By \cite[Theorem 3.1]{powell1989semiparametric}, it follows that $\sqrt{n}(\mathcal{U}_n - \hat{\mathcal{U}}_n) = o_{\PP}(1),$ with $\hat{\mathcal{U}}_n$ denoting the H\'{a}jek projection of $\mathcal{U}_n$. By standard $U$-statistics theory, we have $\hat{\mathcal{U}}_n = O_{\PP}(n^{-1/2})$. Overall, the claim follows. \hfill \qed

\subsection*{Proof of Lemma~\ref{lemmaA1}}
We only prove the statement for $E_{n1}$. The corresponding statement for $E_{n2}$ is obtained from analogous arguments using the regularity assumption \ref{item1}. By means of 
\begin{align*}
    \boldsymbol{e}^*(\F(x)) - \hat{\boldsymbol{e}}(\hat{\F}(x)) &= \int yc(F_0(y), \F(x), \vartheta^*)\dF_0(y) - \int yc(\hat{F}_0(y), \hat{\F}(x), \hat{\vartheta}_n)\dFhat_0(y) \\
    &= \int yc(F_0(y), \F(x), \vartheta^*)\dF_0(y) - \int yc(F_0(y), \F(x), \vartheta^*)\dFhat_0(y) \\
    &+ \int yc(F_0(y), \F(x), \vartheta^*)\dFhat_0(y) - \int yc(\hat{F}_0(y), \hat{\F}(x), \hat{\vartheta}_n)\dFhat_0(y)\,,
\end{align*}
we decompose $E_{n1}$ further into
\begin{align}
    E_{n1} &= E_{n1}^{(1)} + E_{n1}^{(2)}\,, \label{en1+en2}
    \end{align}
where 
\begin{align*}
   E_{n1}^{(1)} &:= \frac{1}{n(n-1)}\sum_{i\not= j}\frac{K_{ij}}{h^d}\Tilde{\varepsilon}_ic_X^*(U_i)\left[\int yc(F_0(y), U_j, \vartheta^*)\dF_0(y)\right. \\
    & \hspace{5.5cm} - \left.\int yc(F_0(y), U_j, \vartheta^*)\dFhat_0(y)\right] \\
   E_{n1}^{(2)}  &:= \frac{1}{n(n-1)}\sum_{i\not= j}\frac{K_{ij}}{h^d}\Tilde{\varepsilon}_ic_X^*(U_i)\left[\int yc(F_0(y), U_j, \vartheta^*)\dFhat_0(y)\right. \\
    & \hspace{5.5cm} - \left.\int yc(\hat{F}_0(y), \hat{U}_j, \hat{\vartheta}_n)\dFhat_0(y)\right] 
\end{align*}
We aim to show that both $E_{n1}^{(1)}$ and $E_{n1}^{(2)}$ are of order $O_{\PP}(1/n)$. Our plan is to directly estimate the mean and variance of $E_{n1}^{(1)},$ and to use Taylor's expansion for $E_{n1}^{(2)}$. We start with the latter. A second-order Taylor expansion for the second factor in $E_{n1}^{(2)}$ yields 
\begin{align*}
    & \int yc(F_0(y), U_j, \vartheta^*)\dFhat_0(y) - \int yc(\hat{F}_0(y), \hat{U}_j, \hat{\vartheta}_n)\dFhat_0(y) \\
    =~& n^{-1}\sum_{k=1}^n Y_k\left[c(F_0(Y_k), U_j, \vartheta^*) - c(\hat{F}_0(Y_k), \hat{U}_j, \hat{\vartheta}_n)\right] \\
    =~& n^{-1} \sum_{k=1}^n Y_k \nabla c(F_0(Y_k), U_j, \vartheta^*)\begin{pmatrix} (F_0 - \hat{F}_0)(Y_k) \\ (\F - \hat{\F})(X_{j}) \\ \vartheta^* - \hat{\vartheta}_n \end{pmatrix} \\
    &\quad - \frac{1}{2}n^{-1}\sum_{k=1}^n Y_k\begin{pmatrix} (F_0 - \hat{F}_0)(Y_k) \\ (\F - \hat{\F})(X_{j}) \\ \vartheta^* - \hat{\vartheta}_n \end{pmatrix}^\top \mathrm{H}c(\widetilde{U}_{0k}, \widetilde{U}_{1k}, \widetilde{\vartheta}_{kn})\begin{pmatrix} (F_0 - \hat{F}_0)(Y_k) \\ (\F - \hat{\F})(X_{j}) \\ \vartheta^* - \hat{\vartheta}_n \end{pmatrix} \\
    =:~& \mathcal{T}_1(Y_1, \ldots, Y_n, U_j, \vartheta^*, \hat{\vartheta}_n) - \mathcal{T}_2(Y_1, \ldots, Y_n, U_j, \vartheta^*, \hat{\vartheta}_n)\,,
\end{align*}
again with intermediate spots $\widetilde{U}_{0k},$ $\widetilde{U}_{1k},$ $\widetilde{\vartheta}_{k}$ between $F_0(Y_k), U_j, \vartheta^*$ and $\hat{F}_0(Y_k), \hat{U}_j, \hat{\vartheta}_n,$ respectively. The quadratic forms in $\mathcal{T}_2$ read, e.g., 
\[
    (F_0 - \hat{F}_0)(Y_k) \partial_{u_0}^2 c(\widetilde{U}_{0k}, \widetilde{U}_{1k}, \widetilde{\vartheta}_{k})(F_0 - \hat{F}_0)(Y_k)\,,
\]
again involving two terms of order $O_{\PP}(1/\sqrt{n}).$ So, by the regularity conditions \ref{item3} and \ref{item4}, and by use of the substitution argument as in \eqref{subst}, the entire second-order remainder term satisfies 
\[
    \frac{1}{n(n-1)} \sum_{i\not=j} \frac{K_{ij}}{h^d} \Tilde{\varepsilon}_i c_X^*(U_i) \mathcal{T}_2(Y_1, \ldots, Y_n, U_j, \vartheta^*, \hat{\vartheta}_n) = O_{\PP}(1/n)\,.
\]
For the first-order term, we first approximate $\mathcal{T}_1(Y_1, \ldots, Y_n, U_j, \vartheta^*, \hat{\vartheta}_n)$ by its expectation under $Y$, that is,
\[
    \mathcal{T}_1(Y_1, \ldots, Y_n, U_j, \vartheta^*, \hat{\vartheta}_n) \approx O_{\PP}\left(\frac{1}{\sqrt{n}}\right)\E_Y\negthinspace\left(Y\sum_{z \in \{u_0, \textbf{u}, \vartheta\}} \partial_z c(F_0(Y), U_j, \vartheta^*) \right)\,.
\]
To show that this leaves an error of $O_{\PP}(1/n),$ we have to justify that
\begin{equation}
    n^{-1}\sum_{k=1}^n Y_k \partial_zc(F_0(Y_k), \F(X_j), \vartheta^*) - \E_Y\bigl(Y\partial_z c(F_0(Y), \F(X_j), \vartheta^*)\bigr) = O_{\PP}(n^{-1/2}) \label{eqn17}
\end{equation}
uniformly in $\F(X_j),$ for each $z$. Similarly as in the proof of Lemma~\ref{lemmauniform}, we put
\[
    f_\textbf{u}(Y) := Y\partial_z c(F_0(Y), \textbf{u}, \vartheta^*)
\]
for any $\textbf{u} \in [0,1]^d$, so that the above expression in \eqref{eqn17} is rewritten as $n^{-1}\sum_{k=1}^n f_\textbf{u}(Y_k) - \E(f_\textbf{u}(Y))$. Due to the regularity assumptions, the partial derivative $\partial_z c$ is Lipschitz continuous in $\textbf{u}$, i.e., for $\textbf{u}_1, \textbf{u}_2 \in [0,1]^d,$ we have $|f_{\textbf{u}_1}(y) - f_{\textbf{u}_2}(y)|\leq yL\|\textbf{u}_1 - \textbf{u}_2\|$ for some constant $L>0$. Again by \cite[Example 19.7]{van2000asymptotic}, the class $\mathfrak{F} := \{f_\textbf{u} \mid \textbf{u} \in [0,1]^d\}$ is Donsker, hence, we conclude \eqref{eqn17}. Therefore, we can focus on
\[
    \frac{1}{n(n-1)}\sum_{i\not=j} \frac{K_{ij}}{h^d} \Tilde{\varepsilon}_i c_X^*(U_i) M(X_j)
\]
for $M(X_j) := \E_Y\left(Y\sum_{z \in \{u_0, \textbf{u}, \vartheta\}} \partial_z c(F_0(Y), \F(X_j), \vartheta^*) \right)$. Since the additional factor $c_X^*(U_i)$ does not cause any issues, we can now use: 

\begin{lemma}[see \cite{zheng1996consistent}, Lemma 3.3b] \label{lemma5.1}
    Given $h = o(1),$ $nh^d \rightarrow \infty$ and the regularity assumptions used throughout \cite{zheng1996consistent}, we have under the null hypothesis that 
    \[
    \mathcal{W}_n := \frac{1}{n(n-1)}\sum_{i\not=j}\frac{K_{ij}}{h^d}\Tilde{\varepsilon}_i M(X_j) = O_{\PP}(1/\sqrt{n})\,,
    \]
    where $M$ is continuous and $\|M(x)\| \leq b(x)$ for a function with $\E(b(X_i)^2) < \infty$. 
\end{lemma}

\noindent Originally, \cite[Lemma 3.3b]{zheng1996consistent} required that $M$ is continuously differentiable, but continuity of $M$ together with the domination property is sufficient. By Lemma~\ref{lemma5.1}, the first-order part
\[
    \frac{1}{n(n-1)} \sum_{i\not=j} \frac{K_{ij}}{h^d} \Tilde{\varepsilon}_i c_X^*(U_i) \mathcal{T}_1(Y_1, \ldots, Y_n, X_j, \vartheta^*, \hat{\vartheta}_n) 
\]
is also $O_{\PP}(1/n),$ completing the proof of 
\begin{equation}
    E_{n1}^{(2)} = O_{\PP}(1/n)\,. \label{en12}
\end{equation}

We now address $E_{n1}^{(1)}$, where we estimate the mean and variance. Starting from 
\begin{align*}
    \E \big [ E_{n1}^{(1)} \big ]  &= \frac{1}{n(n-1)}\sum_{i\not=j} \E\left [ \frac{K_{ij}}{h^d} \Tilde{\varepsilon}_i c_X^*(U_i) \int yc(F_0(y), \F(X_{j}), \vartheta^*)\mathrm{d}(F_0 - \hat{F}_0)(y)\right ] \,,
\end{align*}
and taking 
$$
Z_{jk} := -Y_kc(F_0(Y_k), \F(X_j), \vartheta^*) + \E(Yc(F_0(Y), \F(X_1), \vartheta^*))\,,
$$
and noting that 
$\E\left(\Tilde{\varepsilon}_iZ_{jk} \mid X_{i}, X_{j}\right) =0 ~\text{if}~ k \not= i$, we write this (by use of iterated expectation) as
\begin{align*}
    \E\left [E_{n1}^{(1)}\right ] &= \frac{1}{n(n-1)}\sum_{i\not=j} \E\left[\left.\frac{K_{ij}}{h^d}\E\left(\Tilde{\varepsilon}_i c_X^*(U_i)n^{-1}\sum_{k=1}^n Z_{jk}\,\right| X_{i}, X_{j}\right)\right] \\
    &= \frac{1}{n^2(n-1)} \sum_{i\not=j} \E\left(\frac{K_{ij}}{h^d}\Tilde{\varepsilon}_i c_X^*(U_i)Z_{ji}\right). 
\end{align*}
From the substitution argument (see \eqref{subst}) and the regularity conditions \ref{item3} and \ref{item4}, we again obtain 
\[
    \E(h^{-d}K_{ij}c_X^*(U_i)Z_{ji}) = O(1)\,,
\]
giving $\E\left(E_{n1}^{(1)}\right) = O(1/n)$.
For the variance, we accordingly have 
\begin{align*} 
    &\E\left(\left(E_{n1}^{(1)}\right)^2\right) \\
    =~& \frac{1}{n^2(n-1)^2}\sum_{i\not=j}\sum_{i'\not=j'} \E\biggl[\frac{K_{ij}}{h^d}\frac{K_{i'j'}}{h^d}\Tilde{\varepsilon}_i \Tilde{\varepsilon}_{i'}c_X^*(U_i)c_X^*(U_{i'})  \int yc(F_0(y), \F(X_{j}), \vartheta^*)\ddiff(y) \\
    &\hspace{4cm} \left.\cdot \int yc(F_0(y), \F(X_{j'}), \vartheta^*)\ddiff(y)\right] \\
    =~& O\left(\frac{1}{n^6h^{2d}}\right)\sum_{i\not=j}\sum_{i'\not=j'}\sum_{k,k'} \E\Bigl[K_{ij}K_{i'j'} \E\Bigl(\Tilde{\varepsilon}_i\Tilde{\varepsilon}_{i'}c_X^*(U_i)c_X^*(U_{i'})Z_{jk}Z_{j'k'} \mid X_{i}, X_{i'}, X_{j}, X_{j'}\Bigr)\Bigr].
\end{align*}
The conditional mean $\E\left(\Tilde{\varepsilon}_i\Tilde{\varepsilon}_{i'}c_X^*(U_i)c_X^*(U_{i'})Z_{jk}Z_{j'k'} \mid X_{i}, X_{i'}, X_{j}, X_{j'}\right)$ vanishes if $i \not= i'$ and if either $i \notin \{k, k'\}$ or $i' \notin \{k, k'\}$. In other words, we only consider the case $i=i'$ as well as the cases $i=k, i'=k'$ and $i=k'$, $i'=k$. In conclusion,
\begin{align*}
    &\E\left(\left(E_{n1}^{(1)}\right)^2\right) \\
    =~& O\left(\frac{1}{n^6h^{2d}}\right)\sum_{i\not= j,j'} \E\left[K_{ij}K_{ij'} \E\left(\left.\Tilde{\varepsilon}_i^2 c_X^*(U_i)^2\sum_{k,k'} Z_{jk}Z_{j'k'}\,\right| X_{i}, X_{j}, X_{j'}\right)\right] \\
    +~& O\left(\frac{1}{n^6h^{2d}}\right)\sum_{k\not=k'}\sum_{j\not=k, j'\not=k'} \E\Bigl[K_{kj}K_{kj'}\Tilde{\varepsilon}_k\Tilde{\varepsilon}_{k'}c_X^*(U_k)c_X^*(U_{k'})(Z_{jk}Z_{j'k'} + Z_{j'k}Z_{jk'})\Bigr].
\end{align*}
The second term is $O(n^{-2})$ as it only involves four summation indices, and 
\begin{align*}
\E\Bigl(h^{-2d}K_{kj}K_{kj'}\Tilde{\varepsilon}_k\Tilde{\varepsilon}_{k'}c_X^*(U_k)c_X^*(U_{k'})Z_{jk}Z_{j'k'}\Bigr) = O(1)\,,
\end{align*}
which holds due to the regularity conditions \ref{item3} and \ref{item4}, as well as the previous substitution argument. We rewrite the first term as 
\begin{align*}
    & O\left(\frac{1}{n^4h^{2d}}\right)\sum_{i\not= j,j'} \E\left[K_{ij}K_{ij'} \E\left(\left.\Tilde{\varepsilon}_i^2 c_X^*(U_i)^2\left(\frac{1}{n}\sum_{k=1}^n Z_{jk} \right)\left(\frac{1}{n}\sum_{k'=1}^n Z_{j'k'} \right)\,\right| X_{i}, X_{j}, X_{j'} \right)\right] \\
    \leq~& O\left(\frac{1}{n^4h^{2d}}\right)\sum_{i\not= j,j'} \E\left[K_{ij}K_{ij'}\E\left(\Tilde{\varepsilon}_i^2 \left(\frac{1}{n}\sum_{k \not= j} Z_{jk} + \frac{1}{n} Z_{jj}\right)\right.\right. \\
    & \hspace{5.32cm} \cdot \negthinspace \left.\left.\left.\left(\frac{1}{n}\sum_{k' \not= j'} Z_{j'k'} + \frac{1}{n}Z_{j'j'}\right)\,\right| X_{i}, X_{j}, X_{j'}\right)\right] \\
    =~& O\left(\frac{1}{n^4h^{2d}}\right) \sum_{i\not= j,j'} \E\biggl[K_{ij}K_{ij'}\E\biggl(\Tilde{\varepsilon}_i^2\biggl(O_{\PP}(n^{-1}) + O_{\PP}(n^{-3/2}) (Z_{jj} + Z_{j'j'}) \\
    & \hspace{7.08cm} + \negthickspace \left.\left.\left.\left.\frac{1} {n^2}Z_{jj}Z_{j'j'}\right)\,\right| X_{i}, X_{j}, X_{j'} \right)\right] \\
    =~& O\left(\frac{1}{n^5h^{2d}}\right)\sum_{i\not= j,j'} \E(K_{ij}K_{ij'} \Tilde{\varepsilon}_i^2) = O(n^{-2})\,,
\end{align*}
where we used that $n^{-1}\sum_{k \not= j} Z_{jk} = O_{\PP}(n^{-1/2})$ as the $Z_{jk}$ are i.i.d. with $\E(Z_{jk}) = 0$. In conclusion, $\Var(E_{n1}^{(1)}) = O(n^{-2}).$ By Chebyshev's inequality, we have $E_{n1}^{(1)} = O_{\PP}(1/n)$. The assertion now follows from \eqref{en12} and the decomposition \eqref{en1+en2}. \hfill \qed

\subsection*{Proof of Lemma~\ref{lemma4.2}}
Note that $\M_k(X_j)$ depends not only on $X_j$ but also on $X_k$. By taking conditional means, we have 
\[
    \E\left(K_{ij}\Delta^*(X_i)\left(\M_k(X_j) - \varepsilon_j\right)\right) = \E\left(K_{ij}\Delta^*(X_i)\E\left(\M_k(X_j) \mid X_i, X_j\right)\right),
\]    
since we already see that $\E(-\Delta^*(X_i)\varepsilon_j \mid X_i, X_j) = 0$ for independence reasons. By the definition of $\M_k(X_j)$ as in \eqref{mkx} (plugging $x=X_j$), we first compute 
\begin{align*}
    &\E\left(\left.\sum_{l=1}^d \Bigl[\textbf{1}\{X_{k}^{(l)} \leq X_{j}^{(l)}\} - F_l(X_{j}^{(l)})\Bigr]\partial_{x_l} m^*(X_j)\partial_{u_l} F_l^{-1}(U_j^{(l)}) ~\right|\, X_i,X_j\right) \\
    =~& \sum_{l=1}^d \partial_{x_l} m^*(X_j)\partial_{u_l} F_l^{-1}(U_j^{(l)})\,\E\left(\textbf{1}\{X_{k}^{(l)} \leq X_{j}^{(l)}\} - F_l(X_{j}^{(l)}) \mid X_i, X_j\right) = 0\,,
\end{align*}
where we used the fact that the indices $j\neq k$ are distinct. Likewise, we have 
\begin{align*}
    &\E\left(\left. \int \Bigl[\textbf{1}\{Y_k \leq y\} - F_0(y)\Bigr]c(F_0(y), U_j, \vartheta^*)\dy~\right|\,X_i, X_j\right) \\
    =~& \int c(F_0(y), U_j, \vartheta^*) \E\left(\textbf{1}\{Y_k \leq y\} - F_0(y) \mid X_i, X_j\right)\dy = 0\,,
\end{align*}
as we may interchange integration and expectation due to the regularity conditions. Also, 
\[
\E\left(\eta_k^\top \partial_\vartheta m(X_j, \vartheta^*) \mid X_i, X_j\right) = 0
\]
since $\eta_k$ is independent of $X_i$ and $X_j$. It follows that $\E\left(K_{ij}\Delta^*(X_i)\left( \M_k(X_j) - \varepsilon_j\right)\right) = 0$ whenever $i,j,k$ are pairwise distinct, hence, $\sqrt{n}\,\E(W_{n4}) = 0$. \hfill \qed 

\subsection*{Proof of Lemma~\ref{lemma4.3}}
We recall the notations \eqref{det100} and \eqref{wn4quer} and rewrite $\overline{W}_{n4}$ as
\[
    \overline{W}_{n4} = \frac{1}{n(n-1)}\sum_{i\not=j} \Bigl(\M_i(X_j) - \varepsilon_j\Bigr) \frac{1}{n-1} \sum_{\substack{k=1\\k\not=j}}^n \Delta^*(X_j)p(X_j)\,,
\]
which gives 
\begin{align*}
    &\E\biggl[\Bigl(W_{n4} - \overline{W}_{n4}\Bigr)^2\biggr] \\
    =~& \frac{1}{n^3(n - 1)^2(n-2)}\,\E\left[\Biggl(\sum_{i\not=j} \Bigl(\M_i(X_j) - \varepsilon_j\Bigr)\sum_{\substack{k=1\\k\not=j}}^n \left[\frac{K_{kj}}{h^d}\Delta^*(X_k) - \Delta^*(X_j)p(X_j)\right]\Biggr)^2\right] \\
    =~& \frac{1}{n^3(n-1)^2(n-2)} \sum_{\substack{i_1,k_1\not=j_1 \\ i_2,k_2\not=j_2}} \E\Bigl(\M_{i_1j_1}\M_{i_2j_2}\K_{k_1j_1}\K_{k_2j_2}\Bigr)\,,
\end{align*}
where for any $i, j, k \in \{1,..,n\}$ with $i \not= j$ and $k \not= j$:
\begin{align*}
    \M_{ij} &:= \M_{i}(X_{j}) - \varepsilon_{j}\,, \\
    \K_{kj} &:= h^{-d} K_{kj}\Delta^*(X_k) - \Delta^*(X_j)p(X_j)\,.
\end{align*}
In Lemma~\ref{lemma4.2}, we have proven that $\E(\M_{ij} \mid X_j) = 0$ and $\E\bigl(h^{-d}K_{kj}\Delta^*(X_k)\M_{ij}\bigr) = 0$, from which we conclude
\begin{align*}
    \E\Bigl(\M_{ij}\K_{kj}\Bigr) &= \E\Bigl(\M_{ij}\bigl(h^{-d}K_{kj}\Delta^*(X_{k}) - \Delta^*(X_{j})p(X_{j})\bigr)\Bigr) \\
    &= -\E\Bigl(\M_{ij}\Delta^*(X_{j})p(X_{j})\Bigr) = -\E\Bigl(\E\bigl(\M_{ij} \mid X_{j}\bigr) \Delta^*(X_{j})p(X_{j})\Bigr) = 0\,.
\end{align*}
Therefore, if $i_1, j_1, k_1, i_2, j_2, k_2$ are all distinct, then
\[
    \E\Bigl(\M_{i_1j_1}\M_{i_2j_2}\K_{k_1j_1}\K_{k_2j_2}\Bigr) = \E\Bigl(\M_{i_1j_1}\K_{k_1j_1}\Bigr)\E\Bigl(\M_{i_2j_2}\K_{k_2j_2}\Bigr) = 0\,.
\]
Moreover, the tuples $(i_1, j_1, k_1, i_2, j_2, k_2)$ that consist of at most four distinct indices can be ignored, since their expectations are all finite. For the quintuples $(i_1, j_1, k_1, i_2, j_2, k_2)$ with exactly one pair of colliding indices, we have to show that the associated expectations are all $o(1)$. In light of the preceding arguments, we can ignore the cases $i_1=k_1$ and $i_2=k_2$. 

If $k_1 = k_2,$ we simply have
\begin{align*}
    \E\Bigl(\M_{i_1j_1}\M_{i_2j_2}\K_{k_1j_1}\K_{k_1j_2}\Bigr) &= \E\Bigl(\M_{i_2j_2}\K_{k_1j_1}\K_{k_1j_2}\E(\M_{i_1j_1} \mid X_{j_1}, X_{k_1}, X_{i_2}, X_{j_2})\Bigr) \\
    &= \E\Bigl(\M_{i_2j_2}\K_{k_1j_1}\K_{k_1j_2}\E(\M_{i_1j_1} \mid X_{j_1})\Bigr) = 0\,,
\end{align*}
as the remaining indices $i_1, i_2, j_1, j_2, k_1$ are all assumed to be distinct. In the same way, we can handle all other cases except for the case of $i_1=i_2$. For this particular case, we have
\begin{align*}
    \E\Bigl(\M_{i_1j_1}\K_{k_1j_1}\M_{i_1j_2}\K_{k_1j_2}\Bigr) &= \E\Bigl(\M_{i_1j_1}\M_{i_1j_2}\E(\K_{k_1j_2}\K_{k_1j_1} \mid X_{i_1}, X_{j_1}, X_{j_2})\Bigr) \\
    &= \E\Bigl(\M_{i_1j_1}\M_{i_1j_2}\E(\K_{k_1j_1} \mid X_{j_1})^2\Bigr)\,.
\end{align*}
By the previously used substitution argument, we have that 
\[
    \E(\K_{k_1j_1} \mid X_{j_1}) = \int K(u)\Delta^*(X_{j_1} + hu)p(X_{j_1}+hu)\du - \Delta^*(X_{j_1})p(X_{j_1}) = O(h^2)\,,
\]
giving $\E\Bigl(\M_{i_1j_1}\M_{i_1j_2}\E(\K_{k_1j_1} \mid X_{j_1})^2\Bigr) = O(h^4) = o(1)$\,. In conclusion,
\[
    \E\Bigl(W_{n4} - \overline{W}_{n4}\Bigr) = O(n^{-1}h^4) + O(n^{-2}h^{-d}) = o(n^{-1})\,. 
\]

\subsection*{Proof of Lemma~\ref{lemma4.4}}
Since $\E(\overline{W}_{n4}) = 0$ for the same reasons as in the proof of Lemma~\ref{lemma4.2}, we can write
\begin{align*}
    &\Var(\sqrt{n}\,\overline{W}_{n4}) = \E(n\overline{W}_{n4}^2) \\
    =~& \frac{1}{n(n-1)^2}\sum_{i\neq j} \sum_{r \neq s} \E\Bigl(\bigl(\M_i(X_j) - \varepsilon_j\bigr)\Delta^*(X_j)p(X_j)\bigl(\M_r(X_s) - \varepsilon_s\bigr)\Delta^*(X_s)p(X_s)\Bigr) \\
    =~& \frac{1}{n(n-1)^2}\sum_{(i,j,r,s) \in \I}  \E\Bigl(\bigl(\M_i(X_j) - \varepsilon_j\bigr)\Delta^*(X_j)p(X_j)\bigl(\M_r(X_s) - \varepsilon_s\bigr)\Delta^*(X_s)p(X_s)\Bigr) + o(1)\,,
\end{align*}
where $\I$ encompasses all quadruplets $(i,j,r,s)$ with exactly two indices matching, but $i\not=j$ and $r\not=s$ (since again, all distinct quadruplets are omitted due to independence). In the case of $j \neq s$, all summands involving $\varepsilon_j$ or $\varepsilon_s$ are likewise zero. Putting $\I' := \{(i,j,r,s) \in \I \negmedspace: j \neq s\}$, we can write 
\begin{subequations}
\begin{align}
    \sum_{(i,j,r,s) \in \I'}  &\,\E\Bigl(\bigl(\M_i(X_j) - \varepsilon_j\bigr)\Delta^*(X_j)p(X_j)\bigl(\M_r(X_s) - \varepsilon_s\bigr)\Delta^*(X_s)p(X_s)\Bigr) \nonumber \\
    = \sum_{i\not=j, i\not=s, j\not=s} &\,\E\Bigl(\M_i(X_j)\M_i(X_s)\Delta^*(X_j)\Delta^*(X_s)p(X_j)p(X_s)\Bigr) \label{18a} \\
    + \sum_{i\not=j, i\not=r, j\not=r} &\,\E\Bigl(\M_i(X_j)\M_r(X_i)\Delta^*(X_i)\Delta^*(X_j)p(X_i)p(X_j)\Bigr) \label{18b} \\
    + \sum_{i\not=j, j\not=s, i\not=s} &\,\E\Bigl(\M_i(X_j)\M_j(X_s)\Delta^*(X_j)\Delta^*(X_s)p(X_j)p(X_s)\Bigr)\,. \label{18c}
\end{align}
\end{subequations}
where \eqref{18a}, \eqref{18b} and \eqref{18c} represent the cases of $i=r,$ $i=s$ and $j=r,$ respectively. In analogy with an argument already used in Lemma~\ref{lemma4.3}, we can find that for distinct $i,j,r,$
\begin{align*}
    &\E\Bigl(\M_i(X_j)\M_r(X_i)\Delta^*(X_i)\Delta^*(X_j)p(X_i)p(X_j)\Bigr) \\
    =~& \E\Bigl(\M_i(X_j)\Delta^*(X_i)\Delta^*(X_j)p(X_i)p(X_j)\E\bigl(\M_r(X_i) \mid X_i, X_j\bigr)\Bigr) = 0\,,
\end{align*}
i.e., $\eqref{18b} = 0$, and likewise, $\eqref{18c} = 0$. Only \eqref{18a} cannot be treated this way, which means we overall have
\begin{align*}
    \E(n\overline{W}_{n4}^2) &= \frac{n-2}{n-1}\,\E\bigl(\M_1(X_2)\M_1(X_3)\Delta^*(X_2)\Delta^*(X_3)p(X_2)p(X_3)\bigr) \\
    &+ \hspace{0.5mm} \frac{n-2}{n-1}\,\E\left((\M_1(X_2) - \varepsilon_2)(\M_3(X_2) - \varepsilon_2)\Delta^*(X_2)^2p(X_2)^2\right).
\end{align*}
Expanding the second summand, we have
\begin{align*}
    &\E\bigl((\M_1(X_2) - \varepsilon_2)(\M_3(X_2) - \varepsilon_2)\Delta^*(X_2)^2p(X_2)^2\bigr) \\
    =~& \E\bigl(\M_1(X_2)\M_3(X_2)\Delta^*(X_2)^2p(X_2)^2 + \varepsilon_2^2\Delta^*(X_2)p(X_2)^2\bigr) \\
    &- \E\bigl(\varepsilon_2 (\M_1(X_2) + \M_3(X_2))\Delta^*(X_2)^2p(X_2)^2\bigr)\,.
\end{align*}
Obviously, the second expectation is zero, and we also have 
\[
    \E\bigl(\M_1(X_2)\M_3(X_2)\Delta^*(X_2)^2p(X_2)^2\bigr) = \E\bigl(\Delta^*(X_2)^2p(X_2)^2\M_1(X_2)\E(\M_3(X_2) \mid X_1, X_2)\bigr) = 0\,.
\]
The claim follows. \hfill \qed

\subsection*{Proof of Lemma~\ref{lemma4.5}}
We first write
\begin{align*}
    &\Cov\left(\sqrt{n}\left(\widetilde{W}_{n3}^{(2)} - \widetilde{W}_{n2}^{(2)}\right), \sqrt{n}\widetilde{W}_{n3}^{(3)}\right) \\
    =~& \frac{1}{n(n-1)^2} \Cov\left(\sum_{i\not=j} \frac{K_{ij}}{h^d}\Delta^*(X_i)\bigl[(\hat{m} - m^*)(X_j) - \varepsilon_j\bigr],~\sum_{k \not= l} \frac{K_{kl}}{h^d}\Delta^*(X_k)\Delta^*(X_l)\right).
\end{align*}
For independence reasons, we can ignore the tuples of distinct $i,j,k,l$. Letting $\mathcal{I} := \{(i,j,k,l) \in \{1,\ldots,n\}^4: |\{i,j,k,l\}| \leq 3\},$ we have
\begin{align*}
    &\Cov\left(\sqrt{n}\left(\widetilde{W}_{n3}^{(2)} - \widetilde{W}_{n2}^{(2)}\right), \sqrt{n}\widetilde{W}_{n3}^{(3)}\right) \\
    =~& \frac{1}{n(n-1)^2h^{2d}} \sum_{(i,j,k,l) \in \mathcal{I}} \Cov\Bigl(K_{ij}\Delta^*(X_i)[(\hat{m} - m^*)(X_j) - \varepsilon_j],  K_{kl}\Delta^*(X_k)\Delta^*(X_l)\Bigr).
\end{align*}
By use of Taylor expansion,
\begin{align*}
    (\hat{m} - m^*)(X_j) &= \partial_\vartheta m^*(X_j)^\top(\hat{\vartheta} - \vartheta^*) + o_{\PP}(n^{-1/2}) \\
    &= \frac{1}{n}\sum_{i=1}^n \partial_\vartheta m^*(X_j)^\top\eta_i + o_{\PP}(n^{-1/2})\,.
\end{align*}
Note that the partial derivative exists by the regularity assumptions \ref{item3} and \ref{item4}. Hence,
\begin{align*}
    &\Cov\left(\sqrt{n}\left(\widetilde{W}_{n3}^{(2)} - \widetilde{W}_{n2}^{(2)}\right), \sqrt{n}\widetilde{W}_{n3}^{(3)}\right) \\ 
    =~& \frac{1}{n(n-1)^2h^{2d}} \sum_{(i,j,k,l) \in \mathcal{I}} \Cov\left(K_{ij}\Delta^*(X_i) \left(\frac{1}{n}\partial_\vartheta m^*(X_j)^\top\sum_{r=1}^n \eta_r - \varepsilon_j + o_{\PP}(1)\right),\right. \\
    & \hspace{4.7cm}  K_{kl}\Delta^*(X_k)\Delta^*(X_l)\Biggr).
\end{align*}
The terms involving factors of $\varepsilon_j$ all give zero covariance as the $\varepsilon_j$ are independent from all $X_1, \ldots, X_n$. Likewise, for any $r \notin \{k,l\}$, we have for independence reasons:
\[
    \Cov\left(K_{ij}\Delta(X_i)c_X^*(U_i)\partial_\vartheta m^*(X_j)^\top \eta_r, K_{kl}\Delta(X_k)\Delta(X_l)c_X^*(U_k)c_X^*(U_l)\right) = 0\,.
\] 
Therefore, 
\begin{align*}
    &\Cov\left(\sqrt{n}\left(\widetilde{W}_{n3}^{(2)} - \widetilde{W}_{n2}^{(2)}\right), \sqrt{n}\widetilde{W}_{n3}^{(3)}\right) \\
    =~& \frac{1}{n^2(n-1)^2h^{2d}} \sum_{(i,j,k,l) \in \mathcal{I}} \Cov\left(K_{ij}\Delta(X_i)c_X^*(U_i)\partial_\vartheta m^*(X_j)^\top(\eta_k + \eta_l), \right.\\
    &\hspace{4.75cm}K_{kl}\Delta(X_k)\Delta(X_l)c_X^*(U_k)c_X^*(U_l)\Bigl) \\
    +~& \frac{1}{n(n-1)^2h^{2d}} \sum_{(i,j,k,l) \in \mathcal{I}} \Cov\Bigl(K_{ij}\Delta(X_i)c_X^*(U_i)o_{\PP}(1), K_{kl}\Delta(X_k)\Delta(X_l)c_X^*(U_k)c_X^*(U_l)\Bigr)\,.
\end{align*}
By the standard substitution argument (see \eqref{subst}) and regularity assumptions, the first term is $O(1/n)$ and the second term is $o(1),$ giving the claim. \hfill \qed


\subsection*{Proof of Lemma~\ref{lemma5.2}} Part a) is a special case of \cite[Theorem 1]{miller1972weak}, but we give a separate proof that prepares for b). Employ the sequential H\'{a}jek projection 
\[
    \hat{U}_{\nt} := \frac{2t}{n} \sum_{i=1}^{\nt} [g(X_i) - r] + t^2r\,,
\]
with $g(X_i) := \E(h(X_i, X' \mid X_i)$ for an independent copy $X'$. Then, for $k := \nt \in \{1, \ldots, n\},$ the difference $U_{\nt} - \hat{U}_{\nt} = U_k - \hat{U}_k$ can be rewritten as
\begin{align*}
    U_{k} - \hat{U}_{k} &= \frac{1}{n(n-1)}\sum_{i\neq j}^k [h(X_i,X_j) - g(X_i) - g(X_j) + r] + O_{\PP}(1/n) \\
    &= 
    \frac{2}{n(n-1)} \sum_{i=1}^k S_i + O_{\PP}(1/n)\,,
\end{align*} 
with 
$$
S_i := \sum_{j=1}^{i-1} \,[h(X_i,X_j) - g(X_i) - g(X_j) + r].
$$
Observing that
\begin{align*}
    \E(S_i \mid X_1, \ldots, X_{i-1}) &= \sum_{j=1}^{i-1} \E(h(X_i,X_j) - g(X_i) - g(X_j) + r \mid X_1, \ldots, X_{i-1}) \\
    &= \sum_{j=1}^{i-1} \underbrace{\E(h(X_i,X_j) \mid X_1, \ldots, X_{i-1})}_{=\,g(X_j)} - \sum_{j=1}^{i-1} \underbrace{\E(g(X_i) \mid X_1, \ldots, X_{i-1})}_{=\,\E(g(X_i)) \,=\, r} \\
    &\qquad - \sum_{j=1}^{i-1} \underbrace{\E(g(X_j) \mid X_1, \ldots, X_{i-1})}_{=\,g(X_j)}\,+~(i-1)r = 0\,.
\end{align*}
It follows that  $(S_i)_{i=1,\ldots,n}$ is a martingale difference sequence with respect to the filtration $(\mathcal{F}_i)_{i=1,\ldots,n}$, where
$\mathcal{F}_i$ denotes the 
$\sigma$-field generated  by   $(X_1, \ldots, X_i)$.
Therefore, $\sqrt{n}(U_{k} - \hat{U}_{k})_{k=1,\ldots,n}$ forms a martingale and by Doob's inequality,
\[
    \max_{k=1,\ldots,n} \E(n(U_{k} - \hat{U}_{k})^2) \leq 4\E(n(U_{n} - \hat{U}_{n})^2) = o_{\PP}(1)\,.
\]
In conclusion, $\sqrt{n}(U_{\nt} - \hat{U}_{\nt}) = o_{\PP}(1)$ holds uniformly in $t$. Since each $\hat{U}_{\nt}$ is a partial sum of i.i.d.\ samples of the same random variable, it follows by Donsker's theorem that 
\[
    \{\sqrt{n}(\hat{U}_{\nt} - t^2r)\}_{t \in [0,1]} = \left\{2t\,\frac{1}{\sqrt{n}} \sum_{i=1}^{\nt} [g(X_i) - r]\right\}_{t \in [0,1]} \overset{\D}{\longrightarrow} \{\sigma t B(t)\}_{t \in [0,1]}.
\]
\par \medskip
b): Let $g_n(X_i) := \E(h(X_i, X') \mid X_i)$ and let $\hat{U}_n = 2/n \sum_{i=1}^n [g_n(X_i) - r_n] + r_n$ be the H\'{a}jek projection of $U_n$. By the assumption $\E(h_n(X_1, X_2)^2) = o(n),$ \cite[Theorem 3.1]{powell1989semiparametric} guarantees that $\sqrt{n}(U_n - \hat{U}_n) = o_{\PP}(1)$. By the martingale difference argument above, this also holds for $\sqrt{n}(U_{\nt} - \hat{U}_{\nt})$ uniformly in $t$, where $\hat{U}_{\nt} = 2t/n \sum_{i=1}^{\nt} [g_n(X_i) - r_n] + t^2r_n$. We only need to clarify the limit of $\{\sqrt{n}\,\hat{U}_{\nt}\}_{t \in [0,1]}$. 

For each $n \in \N$ and $k \in \{1, \ldots, n\},$ we put $\xi_{nk} := 2tn^{-1/2}(g_n(X_k) - r_n)$, so that
\[
    \sqrt{n}\Bigl(\hat{U}_{\nt} - t^2r_n\Bigr) = \sum_{k=1}^{\nt} \xi_{nk}\,.
\]
Then, the sequences $(\xi_{n1}, \ldots, \xi_{nn})$ form martingale differences as well, and we have for $\sigma_{nk}^2 := \Var(\xi_{nk})$:
\[
    \sigma_n^2 := \sum_{k=1}^{\nt} \sigma_{nk}^2 = \frac{4\nt}{n}\Var(g_n(X_1)) \longrightarrow t\sigma^2\,.
\]
Then, this triangular array satisfies the Lyapunov condition (and hence, the Lindeberg condition): For each $t \in [\delta, 1]$,
\begin{align*}
    \frac{1}{\sigma_n^4}\sum_{j=1}^{\nt} \E\left(|\xi_{nj} - \E(\xi_{nj})|^4\right) &= \frac{t^4}{(4t\sigma_{n1}^2)^2}\sum_{i=1}^n \E\left(\left(\frac{2}{\sqrt{n}}(g_n(X_k) - r_n\right)^4\right) \\
    &= \frac{t^2}{n^2(\sigma_{n1}^2)^2}  \sum_{j=1}^{\nt} \E((g_n(X_j) - r_n)^4) = o(1)\,,
\end{align*}
which is due to the convergence of $\sigma_{n1}^2$ and the uniform boundedness of fourth central moments. Then, all conditions of \cite[Theorem 18.2]{billingsley1999convergence} are satisfied, and the claim follows. \hfill \qed 

\subsection*{Proof of Lemma~\ref{lemma4.7}}
We employ the decomposition
\[
    W_{\nt} = W_{\nt1} - 2\left(W_{\nt2}^{(1)} - W_{\nt2}^{(2)}\right) + \left(W_{\nt3}^{(1)} - 2W_{\nt3}^{(2)} + W_{\nt3}^{(3)}\right)
\]
with $W_{\nt1}, \ldots, W_{\nt3}^{(3)}$ being the sequential counterparts of $W_{n1}, \ldots, W_{n3}^{(3)}$ as introduced in the proof of Theorem~\ref{thm4.1}, that is, double-sums range only over $1 \leq i \neq j \leq \nt$ and estimated quantities only refer to the subsample $(X_1, Y_1), \ldots, (X_{\nt}, Y_{\nt})$. \par \bigskip \noindent 
\textbf{Step 1:} We first show that Lemma~\ref{lemmauniform} holds uniformly in $t$, i.e., 
\begin{align*}
   \sup_{t \in [\delta, 1]} &\|\boldsymbol{e}(x) - \hat{\boldsymbol{e}}_{\nt}(x)\|_\infty = O_{\PP}(n^{-1/2})\,, \\
   \sup_{t \in [\delta, 1]} & \|c_X(\F(x), \vartheta^*) - c_X(\hat{\F}_{\nt}(x), \hat{\vartheta}_n)\|_\infty = O_{\PP}(n^{-1/2})\,.
\end{align*}
In analogy to the proof of Lemma~\ref{lemmauniform}, we put
\[
    \overline{\boldsymbol{e}}_{\nt}(x) := \frac{1}{\nt} \sum_{i=1}^{\nt} Y_i c(F_0(Y_i), \F(x), \vartheta^*)\,.
\]
Adapting the notation of the proof of Lemma~\ref{lemmauniform}, we can write $\overline{\boldsymbol{e}}_{\nt}(x) = \nt^{-1} \sum_{k=1}^{\nt} f_\textbf{u}(Y_k),$ $\boldsymbol{e}(x) = \E(f_\textbf{u}(Y))$. Then, we have $\sqrt{n}(\overline{\boldsymbol{e}}_{\nt}(x) - \boldsymbol{e}(x)) = n/\nt \alpha_n(f,t) \leq \delta^{-1}\alpha_n(f_\textbf{u},t),$ with
\[
    \alpha_n(f_\textbf{u},t) := \frac{1}{\sqrt{n}} \sum_{i=1}^{\nt} (f_\textbf{u}(Y_k) - \E(f_\textbf{u}(Y)))\,.
\]
Then, $\{\alpha_n(f_\textbf{u}, t)\}_{\textbf{u} \in [0,1]^d, t\in [\delta,1]}$ converges to a Gaussian process in $\mathcal{L}^\infty(\mathfrak{F} \times [\delta,1]),$ with $\mathfrak{F} = \{f_\textbf{u} \mid \textbf{u} \in [0,1]^d\}$. The remaining arguments in the proof of Lemma~\ref{lemmauniform} regarding $\|\overline{\boldsymbol{e}} - \hat{\boldsymbol{e}}\|_\infty$ and $\|c_X - \hat{c}_X\|_\infty$ are only based on Taylor expansions, which work uniformly in $t$. \par 
\bigskip 
\textbf{Step 2:} We replace the product $\hat{c}_X(\hat{U}_{i\nt})\hat{c}_X(\hat{U}_{j\nt})$ with $c_X^*(U_i)c_X^*(U_j)$ in the terms $W_{\nt1},  W_{\nt2}^{(2)}$ and $ W_{\nt3}^{(3)}$, i.e., we demonstrate that 
\begin{equation}
    W_{\nt1} - \widetilde{W}_{\nt1} := W_{\nt1} - \frac{1}{\nt(n-1)}\sum_{1 \leq i\neq j \leq \nt} \frac{K_{ij}}{h^d}\varepsilon_{i}\varepsilon_{j}c_X^*(U_i)c_X^*(U_j) = o_{\PP}(n^{-1/2}) \label{step1}
\end{equation}
uniformly in $t \in [\delta,1]$. Recalling the Taylor expansion employed in the proof of Lemma~\ref{lemmareplace1a}, we again take $g_X(U_i, U_j, \vartheta) = c_X(u_i, \vartheta)c_X(u_j, \vartheta)$ and adapt the notation of \eqref{taylor}. From previously used arguments, we have 
\[
    O_{\PP}\left(\frac{1}{n}\right)\frac{1}{\nt(n-1)}\sum_{1\leq i\neq j\leq \nt} \frac{K_{ij}}{h^d} e_{i}e_{j} \mathrm{H}g_X(U_i, U_j, \vartheta^*) = O_{\PP}\left(\frac{\nt}{n(n-1)}\right) = o_{\PP}\left(\frac{1}{\sqrt{n}}\right)\,,
\]
which holds uniformly in $t \in [\delta,1]$. The first-order terms 
\[
    O_{\PP}\left(\frac{1}{\sqrt{\nt}}\right)\frac{1}{\nt(n-1)}\sum_{1\leq i\neq j\leq \nt} \frac{K_{ij}}{h^d} \varepsilon_{i}\varepsilon_{j}(\partial_{u_i} + \partial_{u_j} + \partial_{\vartheta} )g_X(U_i, U_j, \vartheta^*)
\]
can be represented as $U$-statistics, for which we can apply Lemma~\ref{lemma5.2}b), which means that these $U$-statistics have a sequential limiting object at the $\sqrt{n}$-scale. In conclusion, we can successfully replace $W_{\nt1}$ with $\widetilde{W}_{\nt1}$. The same arguments apply for $W_{\nt2}^{(2)}$ and $W_{\nt3}^{(3)}$.
\par 
\bigskip
\textbf{Step 3:} Now, we can figure that the parts $\widetilde{W}_{\nt1}, W_{\nt2}^{(1)}, W_{\nt3}^{(1)}$ are still negligible at the $\sqrt{n}$-scale. Regarding $\widetilde{W}_{\nt1}$, the kernel $H_n(Z_i, Z_j) = h^{-d}K_{ij}\varepsilon_i\varepsilon_j$ still forms a degenerate $U$-statistic, and regarding $W_{\nt2}^{(1)}$, all arguments in the proof of Lemma~\ref{lemmaA1} can be stated uniformly in $t \in [\delta, 1]$ with the aid of Step 1 and Step 2. The negligibility of $W_{\nt3}^{(1)}$ follows directly from Step 1. By analogy with \eqref{eqn5}, we therefore employ the decomposition \[
    \sqrt{n}W_{\nt} = 2\sqrt{n}\left(\widetilde{W}_{\nt2}^{(2)} - \widetilde{W}_{\nt3}^{(2)}\right) + \sqrt{n}\widetilde{W}_{\nt3}^{(3)} + O_{\PP}\left(\frac{1}{\sqrt{n}}\right).
\]
Furthermore, the uniform expansion 
\[
    \hat{m}_{\nt}(x) - m^*(x) = c_X(\F(x))^{-1}\sum_{k=1}^{\nt} \M_k(x) + o_{\PP}(n^{-1/2})
\]
can be understood by previously used arguments, following the original proof of \cite[Theorem 1]{Noh2013}. For $\widetilde{W}_{\nt2}^{(2)} - \widetilde{W}_{\nt3}^{(2)}$, we thus employ the expansion 
\[
    \sqrt{n}\Bigl(\hat{m}_{\nt}(X_j) - m(X_j)\Bigr) = \frac{1}{\sqrt{\nt}}\frac{1}{c_X^*(U_j)}\sum_{k=1}^{\nt} \M_k(X_j)
\]
and ignore the colliding summands (i.e. $k=i$ or $k=j$) in the representation 
\[
    \sqrt{n}\left(\widetilde{W}_{\nt3}^{(2)} -  \widetilde{W}_{\nt2}^{(2)}\right) = \frac{\sqrt{n}}{\nt(n-1)}\sum_{1\leq i\not=j\leq \nt} \frac{K_{ij}}{h^d}\Delta^*(X_i)\frac{1}{\nt}\sum_{k=1}^{\nt} (\M_k(X_j) - \varepsilon_j) + o_{\PP}(1)\,,
\]
i.e., we consider
\[
    \sqrt{n}\,W_{\nt4} = \frac{\sqrt{n}}{\nt^2(n-1)} \sum_{i\neq j\neq k}^{\nt}  \underbrace{(\M_i(X_j) - \varepsilon_j)}_{=:\, f(X_i,X_j)}\underbrace{\frac{K_{kj}}{h^d} \Delta^*(X_k)}_{=:\,k(X_j,X_k)}\,.
\]
The lemmas~\ref{lemma4.2} and~\ref{lemma4.5} remain valid in the sequential setting as well. \par 
\bigskip
\textbf{Step 4:} We replicate Lemma~\ref{lemma4.3}, i.e., we replace $\sqrt{n}\,W_{\nt4}$ with 
\[\sqrt{n}\,\overline{W}_{\nt4} = \frac{\sqrt{n}}{\nt(n-1)} \sum_{i\neq j}^{\nt} (\M_i(X_j) - \varepsilon_j) \underbrace{\Delta^*(X_j) p(X_j)}_{=:\,\overline{k}(X_j)}\,.\]
The difference $W_{\nt4} - \overline{W}_{\nt4}$ is then written as
\begin{align*}
    W_{\nt4} - \overline{W}_{\nt4} &= \frac{1}{\nt^2(n-1)} \sum_{i\neq j \neq k}^{\nt}  f(X_i,X_j)\left(k(X_j,X_k) - \overline{k}(X_j)\right). 
\end{align*}
Now, we take the following sequence of martingale differences: 
\[
    S_i := \sum_{j\neq k}^{i-1} f(X_i,X_j)\left(k(X_j,X_k) - \overline{k}(X_j)\right),
\]
since
\begin{align*}
    &\E(S_{i} \mid X_1, \ldots, X_{i-1}) \\
    =~& \sum_{j\neq k}^{i-1} \E\left((\M_i(X_j) - \varepsilon_j)\left(\left.\frac{K_{kj}}{h^d}\Delta^*(X_k) - \Delta^*(X_j)p(X_j)\right) \right| X_j,X_k\right) \\
    =~& \sum_{j\neq k}^{i-1} \left(\frac{K_{kj}}{h^d}\Delta^*(X_k) - \Delta^*(X_j)p(X_j)\right)\underbrace{\E(\M_i(X_j) - \varepsilon_j \mid X_j, X_k)}_{=\, 0} = 0\,.
\end{align*}
Therefore,
\[
    \frac{\nt^2}{n^2}\left(W_{\nt4} - \overline{W}_{\nt4}\right) = \frac{1}{n^2(n-1)}\sum_{i=1}^n S_i
\] 
forms a martingale and Doob's maximum inequality gives
\begin{align*}
    \E\left(\sup_{t \in [\delta,1]} n\left(W_{\nt,4} - \overline{W}_{\nt,4}\right)^2\right) &\leq \, \E\left(\max_{r=1,\ldots,n} n\left(W_{r,4} - \overline{W}_{r,4}\right)^2\right) \\
    &\leq \frac{4}{\delta^2}\,\E\left(n(W_{n4} - \overline{W}_{n4})^2\right) = o(1)\,.
\end{align*}

\textbf{Step 5:}  
Overall, we have now demonstrated that 
\[
    \{\sqrt{n}(W_{\nt} - \mu_{\nt})\}_{t \in [\delta,1]} \overset{\D}{\longrightarrow} \left\{\sqrt{n}\left(W_{\nt3}^{(3)} - \mu_{\nt} + \overline{W}_{\nt4}\right)\right\}_{t \in [\delta,1]} + o_{\PP}(1)\,,
\]
where $\mu_{\nt} = \E\left(\widetilde{W}_{\nt3}^{(3)}\right) = t\mu_n + O(n^{-1})$. Applying the rescalings
\begin{align*}
    t\,\widetilde{W}_{\nt3}^{(3)} &= \frac{1}{n(n-1)}\sum_{i\neq j}^{\nt} H_n(Z_i, Z_j)\,, & t\,\overline{W}_{n4} &= \frac{1}{n(n-1)}\sum_{i\neq j}^{\nt} \HH(Z_i, Z_j)\,, \end{align*}
we obtain the sequential approximation 
\[
    \left\{\sqrt{n}\,\overline{W}_{\nt4}\right\}_{t \in [\delta,1]} = \left\{\frac{2}{\sqrt{n}}\sum_{i=1}^{\nt} g(Z_i)\right\}_{t \in [\delta,1]} + o_{\PP}(1)
\]
from Lemma~\ref{lemma5.2}a) and the sequential approximation
\[
    \left\{\sqrt{n}\left(\widetilde{W}_{\nt3}^{(3)} - t\mu_n\right)\right\}_{t \in [\delta,1]} = \left\{\frac{2}{\sqrt{n}}\sum_{i=1}^{\nt} [g_n(X_i) - \mu_n]\right\}_{t \in [\delta,1]} + o_{\PP}(1)
\]
from Lemma~\ref{lemma5.2}b). The claim follows. \hfill \qed 
 
\bigskip

{\bf Acknowledgements} Philip Dörr would like to thank Patrick Bastian, Martin Dunsche, Marius Kroll, and Thomas Lam for helpful discussions. 

This work was supported by TRR 391 \textit{Spatio-temporal Statistics for the Transition of Energy and Transport} (Project number 520388526) funded by the Deutsche Forschungsgemeinschaft (DFG, German Research Foundation). 

Simulations (or parts of them) for this publication were performed on the HPC cluster Elysium of the Ruhr University Bochum, subsidized by the DFG (INST 213/1055-1).
\end{appendices}

\bibliographystyle{apalike}
\bibliography{bibliography.bib}

\end{document}